\documentclass[3p,sort&compress,11pt]{elsarticle}
\usepackage{amsmath}

\usepackage{amsfonts}
\usepackage{graphics}
\usepackage{amssymb}

\usepackage{color}
\usepackage{hyperref}

\usepackage[nodots]{numcompress}

\usepackage{natbib}
\usepackage{geometry}
\usepackage{fleqn}
\usepackage{graphicx}
\usepackage{txfonts}

\newtheorem{theorem}{Theorem}[section]
\newtheorem{proposition}{Proposition}[section]
\newtheorem{lemma}{Lemma}[section]

\newdefinition{remark}{Remark}[section]
\newtheorem{definition}{Definition}[section]

\newproof{proof}{Proof}

\newcommand{\e}{^\varepsilon}
\newcommand{\eps}{{\varepsilon}}
\newcommand{\ds}{\displaystyle}

\newcommand{\M}{\mathbf{M}}
\newcommand{\I}{\mathcal{I}\e}
\newcommand{\g}{\gamma}

\newcommand{\dist}{\mathrm{dist}}

\renewcommand{\a}{\alpha}
\renewcommand{\b}{\beta}

\newcommand{\cupl}{\bigcup\limits}
\newcommand{\supp}{\mathrm{supp}}
\newcommand{\suml}{\sum\limits}
\newcommand{\intl}{\int\limits}
\newcommand{\liml}{\lim\limits}
\newcommand{\maxl}{\max\limits}
\newcommand{\minl}{\min\limits}

\renewcommand{\phi}{\varphi}

\numberwithin{equation}{section}

\begin{document}

\begin{frontmatter}
\journal{Journal of Differential Equations}
\author{Andrii Khrabustovskyi}
\title{Periodic Riemannian manifold with preassigned gaps in spectrum of Laplace-Beltrami
operator} \ead{andry9@ukr.net}
\address{Mathematical Division, B. Verkin Institute for Low
Temperature Physics and Engineering of the National Academy of
Sciences of Ukraine, Lenin avenue 47, Kharkiv 61103, Ukraine}

\begin{abstract}
It is known (E.L. Green (1997), O. Post (2003)) that for an
arbitrary $m\in\mathbb{N}$ one can construct a periodic
non-compact Riemannian manifold $M$ with at least $m$ gaps in the
spectrum of the corresponding Laplace-Beltrami operator\ \
$-\Delta_M$. In this work we want not only to produce a new type
of periodic manifolds with spectral gaps but also to control the
edges of these gaps. The main result of the paper is as follows:
for \textcolor{black}{ arbitrary} pairwise disjoint intervals
$(\a_j,\b_j)\subset[0,\infty)$, $j=1,\dots,m$ ($m\in\mathbb{N}$),
for an arbitrarily small $\delta>0$ and for an arbitrarily large
$L>0$ we construct a periodic non-compact Riemannian manifold $M$
with at least $m$ gaps in the spectrum of the operator
$-\Delta_{M}$, moreover the edges of the first $m$ gaps belong to
$\delta$-neighbourhoods of the edges of the intervals
$(\a_j,\b_j)$, while the remaining gaps (if any) are located
outside the interval $[0,L]$.
\end{abstract}


\end{frontmatter}

\section*{Introduction}
In this paper we deal with non-compact periodic manifolds. The
$n$-dimensional Riemannian manifold $M$ is called \textit{periodic}
if there is a discrete finitely generated abelian group $\Gamma$
acting isometrically, properly discontinuously and co-compactly on
$M$. Roughly speaking $M$ is glued from countably many copies of
some compact manifold $\M$ (period cell) and each
$\gamma\in\Gamma$ maps $\M$ to one of these copies.

Let $M$ be an $n$-dimensional periodic Riemannian manifold. We
denote by\ $-\Delta_M$ the Laplace-Beltrami operator on $M$. It is
known (see e.g. \citep{Post_PHD}) that the spectrum
$\sigma(-\Delta_M)$ of the operator $-\Delta_M$ has band-gap
structure, that is
\begin{gather}\label{bands}
\sigma(-\Delta_M)=\cupl_{k=1}^\infty \mathcal{J}_k(M),
\end{gather}
where $\mathcal{J}_k(M)=[a_k,b_k]\subset [0,\infty)$ are compact
intervals called \textit{bands},
$a_k,b_k\underset{k\to\infty}\nearrow\infty$, $a_1=0$. In general
the bands may overlap. The open interval $(\a,\b)$ is called a
\textit{gap} if $(\a,\b)\cap \sigma(-\Delta_M)=\varnothing$ and
$\a,\b\in \sigma(-\Delta_M)$.

The existence of gaps in the spectrum is not guaranteed: for
instance the spectrum of the operator
$-\Delta_{\mathbb{R}^n}=-\suml_{j=1}^n\ds{\partial^2/\partial
x^2_j}$ in $\mathbb{R}^n$ coincides with $[0,\infty)$. It is easy
to see (cf. \citep{DavHar}) that in $1$-dimensional case any
periodic Laplace-Beltrami operator has no gaps. However in the
case $n\geq 2$ we have essentially another situation. Namely, E.
B. Davies and E. M. Harrell II \citep{DavHar} considered the
manifold $M=\mathbb{R}^n$ ($n\geq 2$) with a periodic conformally
flat \textcolor{black}{metric} $g_{ij}=a\delta_{ij}$, where $a=a(x)$ is a periodic
strictly positive smooth function. The authors proved that $a(x)$
can be chosen in such a way that at least one gap in the spectrum
of the operator $-\Delta_M$ exists.

Further, E. L. Green \citep{Green} for any $m\in \mathbb{N}$
constructed a periodic conformally flat \textcolor{black}{metric} in $\mathbb{R}^2$
such that the corresponding Laplace-Beltrami operator has at least
$m$ gaps in the spectrum.

Manifolds of another type were studied by O. Post in
\citep{Post_JDE}, \textcolor{black}{where the author} considered two
different constructions: first, he constructed a periodic manifold
$M\e$ ($\eps>0$ is a small parameter) starting from countably many
copies of a fixed compact manifold connected by small cylinders
(the parameter $\eps$ characterizes a size of the cylinders), in
the second construction he started from a periodic manifold which
further is conformally deformed (the parameter $\eps$
characterizes sizes of domains where the \textcolor{black}{metric}
is deformed). For any $m\in \mathbb{N}$ the existence of $m$ gaps
is proved for $\eps$ small enough. These results were generalized
by F. Lledo and O. Post \citep{Lledo} to the case of periodic
manifolds with non-abelian group $\Gamma$.

Also P. Exner and O. Post \citep{Exner} proved the existence of
gaps for some graph-like manifolds, i.e. the manifolds which
shrink with respect to an appropriate parameter to a graph.

We remark that \textcolor{black}{a similar problem} (i.e. the
existence of gaps in the spectrum) was studied in
\citep{Figotin1,Frielander,Hempel,Zhikov} for periodic divergence
type elliptic operators in $\mathbb{R}^n$, in \citep{HempelHerbst}
for periodic magnetic Schr\"{o}dinger operator, and in
\citep{Figotin2,Filonov} for periodic Maxwell operator. In these
works the gaps in the spectrum are the consequence of a high
contrast in the coefficients. We refer to the overview
\citep{HempelPost} where these and other related  questions are
discussed in detail.

In the present work we want not only to construct a new type of
periodic Riemannian manifolds with gaps in the spectrum of the
Laplace-Beltrami operator but also be able \textit{to control the
edges of these gaps}. Namely the goal of the work is to solve the
following problem: for an arbitrary finite set of pairwise
disjoint finite intervals on the positive semi-axis to construct a
periodic Riemannian manifolds $M$ with at least $m$ gaps in the
spectrum of $-\Delta_M$ (here $m$ is the number of the preassigned
intervals), moreover the first $m$ gaps have to be "close" to the
preassigned intervals, and the remaining gaps (if any) have to be
"close" to infinity.

Let us formulate the main result of the paper.

\begin{theorem}[Main Theorem]\label{th1}
Let $(\a_j,\b_j)\subset[0,\infty)$ ($j={1,\dots,m},\ m\in
\mathbb{N} $) be arbitrary pairwise disjoint finite intervals. Let
$\delta>0$ be \textcolor{black}{an arbitrarily} small number, $L>0$
be \textcolor{black}{an arbitrarily} large number. Let
$n\in\mathbb{N}\setminus\{1\}$.

Then there exists an $n$-dimensional periodic Riemannian manifold
$M$, which can be constructed in the explicit form, such that
\begin{gather}\label{spec1}
\sigma(-\Delta_{M})=[0,\infty)\setminus
\left(\cupl_{j=1}^{m^\prime}(\a_j^\delta,\b_j^\delta)\right),\quad
m\leq m^\prime\leq\infty
\end{gather}
where $(\a_j^\delta,\b_j^\delta)\subset[0,\infty)$ are pairwise
disjoint finite intervals satisfying
\begin{gather}\label{spec2}\begin{matrix}
|\a_j^\delta-\a_j|+|\b_j^\delta-\b_j|<\delta,& j={1,\dots,m}\\
(\a_j^\delta,\b_j^\delta)\subset(L,\infty),&
j={m+1,\dots,m^\prime}\end{matrix}
\end{gather}
\end{theorem}

\begin{remark}
\textcolor{black}{In 1987 Y. Colin de Verdi\`{e}re obtained the
following remarkable result \citep{CDV1}: for arbitrary numbers
$0=\lambda_1<\lambda_2<\dots<\lambda_m$ ($m\in\mathbb{N}$) and
$n\in\mathbb{N}\setminus\left\{1\right\}$ there exists a
$n$-dimensional \textit{compact} Riemannian manifold $M$ such that
the first $m$ eigenvalues of the corresponding Laplace-Beltrami
operator $-\Delta_M$ are exactly
$\left\{\lambda_j\right\}_{j=1}^m$. Our main theorem can be
regarded as an analogue of this result for the case of
\textit{non-compact periodic} Riemannian manifolds.}
\end{remark}

\begin{remark}
Obviously it is sufficient to prove Theorem \ref{th1} only for
such intervals $(\a_j,\b_j)$ that are \textit{nonvoid} and their
\textit{closures} are pairwise disjoint and belong to
$(0,\infty)$. For definiteness we renumber the intervals in the
increasing order, i.e.
\begin{gather}\label{intervals+}
0<\a_1,\quad \a_j<\b_{j}< \a_{j+1},\ j=\overline{1,m-1},\quad
\a_m<\b_{m}<\infty
\end{gather}
\textcolor{black}{Proving} Theorem \ref{th1} we suppose that the
intervals $(\a_j,\b_j)$ satisfy (\ref{intervals+}).
\end{remark}

The idea how to construct the manifold $M$ comes from one of the
directions in the theory of homogenization of PDE's (for classical
problems of the homogenization theory we refer e.g. to the
monographs \citep{March,Tartar,ZKO}). This direction deals with
\textcolor{black}{problems of the following type}. Let $M\e$ be a
Riemannian manifold depending on a small parameter $\eps$: it
consists of one or several copies of some fixed manifold (we call
it "basic manifold") with many attached small surfaces whose
number tends to infinity as $\eps\to 0$. On $M\e$ some PDE (heat
equation, wave equation, Maxwell equations etc.) is considered.
The problem is to describe the behaviour of its solutions as
$\eps\to 0$. More exactly the problem is to find the equation on
the basic manifold (so-called "homogenized equation") whose
solutions approximate the solutions of the pre-limit equation as
$\eps\to 0$.

Firstly the problem of this type was studied by L.Boutet de Monvel
and E.Ya. Khruslov in \citep{Khrus1} where the behaviour of the
diffusion equation was investigated. The asymptotic behaviour of
the spectrum of the Laplace-Beltrami operator was studied in
\citep{DalMaso,Khrab3,Khrab4,Khrab5,Khrab6}, in these works only
compact manifolds were considered.

Let us describe briefly the construction of the manifold $M$
solving our main problem. We denote by $\Omega\e$ ($\eps$ is a
small parameter) a non-compact domain which is obtained by
removing from $\mathbb{R}^n$ a countable set of pairwise disjoint
balls $D_{ij}\e$ ($i\in\mathbb{Z}^n$, $j=1,\dots,m$). It is
supposed that $D_{ij}\e=D_{0j}\e+\eps i$ and
$D_{0j}\e\subset\square_0\e=\left\{x\in\mathbb{R}^n:\ 0\leq
x_\a\leq\eps,\ \forall\a\right\}$. We denote by $d_j\e$ the radius
of the ball $D_{ij}\e$.  Let $B_{ij}\e$ ($i\in\mathbb{Z}^n$,
$j=1,\dots,m$) be an $n$-dimensional surface (we call it "bubble")
obtained by removing a small segment from the $n$-dimensional
sphere of the radius $b_{j}\e$. Identifying the points of
$\partial D_{ij}\e$ and $\partial B_{ij}\e$ we glue the bubbles
$B_{ij}\e$ ($i\in\mathbb{Z}^n$, $j=1,\dots,m$) to the domain
$\Omega\e$ and obtain the $n$-dimensional manifold $M\e$:
\begin{gather*}
M\e={\Omega\e}\cup\left(\cupl_{i\in
\mathbb{Z}^n}\cupl_{j=1}^m{B_{ij}\e}\right)
\end{gather*}
The manifold $M\e$ (for $m=2$) is presented on the Figure
\ref{fig1}.

\begin{figure}[h]
\begin{center}
\scalebox{0.9}[0.9]{\includegraphics{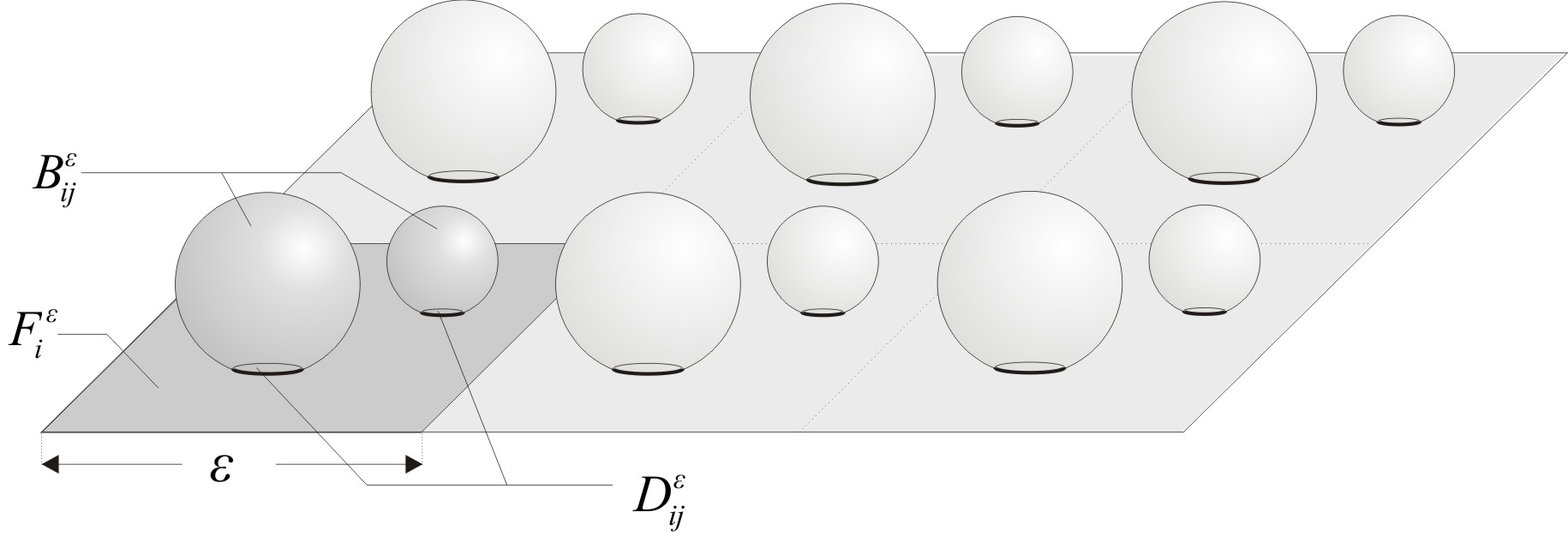}}
 \caption{\label{fig1}The manifold $M\e$ ($m=2$). The period cell $\M_i\e$ is tinted in more dark colour.}
\end{center}
\end{figure}

We \textcolor{black}{equip $M\e$ with} the Riemannian
\textcolor{black}{metric} $g\e$ which coincides with the flat
Euclidean \textcolor{black}{metric} in $\Omega\e$ and coincides with
the spherical \textcolor{black}{metric} on the bubbles $B_{ij}\e$.

The manifold $M\e$ is periodic, the set $$\M_i\e=F_i\e\cup
\left(\cupl_{j=1}^m B_{ij}\e\right)\text{, where }
F_i\e=\square_0\e\setminus\left(\cupl_{j=1}^m D_{0j}\e\right)+\eps
i$$ is a period cell (for any $i\in\mathbb{Z}^n$).

We set $d_j\e=d_j\eps^{{n\over n-2}}$ if $n>2$ and $
d_j\e=\exp\left(-{1\over d_j\eps^2}\right)$ if $n=2$,
$b_{j}\e=b_j\eps$. Here $d_j$,\ $b_j$ ($j=1,\dots,m$) are some
positive constants which will be chosen later.

We prove (see Theorem \ref{th2}) that the spectrum
$\sigma(-\Delta_{M\e})$ of the operator $-\Delta_{M\e}$ has at
least $m$ gaps when $\eps$ is small enough (i.e. when $\eps$ is
less than some $\eps_0$). We denote by $(\sigma_j\e,\mu_j\e)$
($j=1,\dots,m$) the first $m$ gaps, \textcolor{black}{by
$\mathcal{J}\e$ we denote the union of the remaining gaps (if
any)}:
\begin{gather}\label{req0}\sigma(-\Delta_{M\e})=
[0,\infty)\setminus\left[\left(\cupl_{j=1}^m(\sigma_j\e,\mu_j\e)\right)\cup
\mathcal{J}\e\right]\end{gather} Then
\begin{gather}\label{req1}
\forall j=1,\dots,m:\quad \liml_{\eps\to
0}\sigma_j\e=\sigma_j,\quad \liml_{\eps\to 0}\mu_j\e=\mu_j\\
\label{req2} \liml_{\eps\to 0}\inf\mathcal{J}\e=\infty
\end{gather}
where the numbers $\sigma_j$, $\mu_j$ depend in a special way on
$d_j$, $b_j$ and satisfy the conditions
\begin{gather*}
0<\sigma_1,\quad \sigma_{j}<\mu_{j}<\sigma_{j+1},\
j=\overline{1,m-1},\quad \sigma_m<\mu_m<\infty
\end{gather*}
The set $
[0,\infty)\setminus\left(\cupl_{j=1}^m(\sigma_j,\mu_j)\right)$
coincides with the spectrum of some operator $\mathcal{A}$ acting
in the Hilbert space
$H=L_2(\mathbb{R}^n)\underset{j=\overline{1,m}}\oplus
L_2(\mathbb{R}^n,\rho_j dx)$, where $\rho_j$ ($j=1,\dots,m$) are
some positive constant weights, by $dx$ we denote the density of
the Lebesgue measure.

\begin{remark}\label{rem00} \textcolor{black}{In the case when $\Omega\e$
is obtained by removing a system of balls from some
\textit{compact} domain $\Omega$ and $m=1$ (i.e. the removed balls
are equivalent, the attached bubbles are also equivalent)} the
behaviour of the spectrum of the Laplace-Beltrami operator with
Dirichlet boundary conditions on $\partial M\e=\partial\Omega$ was
studied in \citep{Khrab3}, also it was studied in \citep{Khrab5}
for another size of the removed balls, namely $\eps^{{n\over
n-2}}\ll d_j\e\ll\eps$ if $n>2$ and $ \exp\left(-{1\over
a\eps^2}\right)\ll d_j\e\ll\eps$ ($\forall a>0$) if $n=2$. The
same manifolds were also considered in \citep{Khrus2} where the
behaviour of attractors for semi-linear parabolic equations was
investigated.

It was proved in \citep{Khrab3} that the spectrum of the operator
$-\Delta_{M\e}^D$ (here $^D$ means the Dirichlet boundary
conditions) converges in the Hausdorff sense (see the definition
at the beginning of Section \ref{sec3}) to the spectrum of some
self-adjoint operator $\mathcal{A}^D$ acting in the space
$L_2(\Omega)\oplus L_2(\Omega,\rho dx)$, where $\rho>0$ is some
constant weight. The spectrum $\sigma(\mathcal{A})$ of the
operator $\mathcal{A}^D$ has the form
\begin{gather*}
\sigma(\mathcal{A}^D)=\big\{\sigma\big\}\cup\left\{\lambda_k^{D,-}:\
k=1,2,3...\right\} \cup\left\{\lambda_k^{D,+}:\ k=1,2,3...\right\}
\end{gather*}
where  $\sigma>0$ is a point of the essential spectrum, the
nondecreasing sequences $\lambda^{D,-}_k$, $\lambda^{D,+}_k$
belong to the discrete spectrum, moreover
$\liml_{k\to\infty}\lambda^{D,-}_k=\sigma$,
$\liml_{k\to\infty}\lambda^{D,+}_k=\infty$ and
$\lambda^{D,+}_1>\mu$, where $\mu=\sigma+\sigma\rho$.
\textcolor{black}{Thus,}
$(\sigma,\mu)\cap\sigma(\mathcal{A}^D)=\varnothing$,
and\textcolor{black}{, therefore,} for an arbitrarily small
$\delta>0$ the interval $(\sigma+\delta,\mu-\delta)$ does not
intersect with the spectrum of the operator $-\Delta_{M\e}^D$
\textcolor{black}{when $\eps=\eps(\delta)$ is small enough}. A
similar result is valid for the Neumann Laplacian
$-\Delta_{M\e}^N$: the spectrum of the corresponding limit
operator $\mathcal{A}^N$ consists of the point $\sigma$ and two
nondecreasing sequences $\lambda^{N,-}_k$, $\lambda^{N,+}_k$ such
that $\liml_{k\to\infty}\lambda^{N,-}_k=\sigma$,
$\liml_{k\to\infty}\lambda^{N,+}_k=\infty$. Moreover
$\lambda_1^{N,+}=\mu$. It is important that  $\sigma$, $\rho$ are
independent of the shape of the domain $\Omega$ and the type of
the boundary conditions. These facts suggest that in the case
$\Omega=\mathbb{R}^n$ the spectrum $\sigma(-\Delta_{M\e})$ has a
gap \textcolor{black}{when $\eps$ is small enough}, and this gap is
close to the interval $(\sigma,\mu)$.
\end{remark}

The proof of Theorem \ref{th2} consists of three steps. Firstly we
prove that the set $
[0,\infty)\setminus\left(\cupl_{j=1}^m(\sigma_j,\mu_j)\right)$
coincides with the spectrum $\sigma(\mathcal{A})$ of the operator
$\mathcal{A}$. Then we make the main step: we show that for an
arbitrary $L\notin\cupl_{j=1}^m\left\{\mu_j\right\}$ the set
$\sigma(-\Delta_{M\e})\cap[0,L]$ converges in the Hausdorff sense
to the set $\sigma(\mathcal{A})\cap[0,L]$ as $\eps\to 0$.
\textcolor{black}{Finally,} we prove that within an arbitrary finite
interval $[0,L]$ the spectrum $\sigma(-\Delta_{M\e})$ has
\textcolor{black}{at most $m$ gaps} \textcolor{black}{when $\eps$ is
small enough}. Together with the Hausdorff convergence this fact
will imply the properties (\ref{req0})-(\ref{req2}) (see
Proposition \ref{prop1} at the beginning of Section \ref{sec3}).

We note that the \textcolor{black}{metric} $g\e$ is continuous but piecewise-smooth.
However one can approximate it by a smooth \textcolor{black}{metric} $g^{\eps\rho}$
that differs from $g\e$ only in a small $\rho$-neighbourhoods of
$\partial B_{ij}\e$. Moreover when $\rho=\rho(\eps)$ is
sufficiently small then the spectra of the operator
$-\Delta_{(M\e,g^{\eps\rho})}$ and the operator $-\Delta_{M\e}$
have the same limit as $\eps\to 0$ (here
$-\Delta_{(M\e,g^{\eps\rho})}$ is the Laplace-Beltrami operator on
$M\e$ equipped with the \textcolor{black}{metric} $g^{\eps\rho}$). For precise
statement see Remark \ref{rem22} at the end of the paper.

In order to omit cumbersome calculations further we will work with
the \textcolor{black}{metric} $g\e$.

Now, let $\delta>0$ be arbitrarily small number, $L>0$ be
arbitrarily large number. It follows from Theorem \ref{th2} that
there is such small $\eps=\eps(\delta,L)$ that the structure of
the spectrum $\sigma(-\Delta_{M\e})$ is as follows:
$\sigma(-\Delta_{M\e})$ has $m$ gaps whose edges are located in
$\delta$-neighbourhoods of the edges of some fixed intervals
$(\sigma_j,\mu_j)$ ($j=1,\dots,m$) while the remaining gaps (if
any) belong to $(L,\infty)$. So we set $M=M\e$,
$\eps=\eps(\delta,L)$. In order to continue the proof of Theorem
\ref{th1} we have to prove that for \textcolor{black}{ arbitrary}
preassigned intervals $(\a_j$, $\b_j)$ satisfying
(\ref{intervals+}) it is possible to choose such $d_j$, $b_j$ that
\begin{gather}\label{ab}
\sigma_j=\a_j,\ \mu_j=\b_j,\ j=\overline{1,m}
\end{gather}
We will prove this fact and present the exact formulae for the
constants $d_j$, $b_j$ (see Theorem \ref{th3}).

The paper is organized as follows. In Section \ref{sec1} we recall
some definitions and facts from the spectral theory for the
Laplace-Beltrami operator. In Section \ref{sec2} we construct the
manifold $M\e$ and formulate Theorem \ref{th2} describing the
behaviour of $\sigma(-\Delta_{M\e})$ as $\eps\to 0$. Theorem
\ref{th2} is proved in Section \ref{sec3}. And\textcolor{black}{,
finally, }in Section \ref{sec4} we present the formulae for the
parameter $d_j$, $b_j$.

\section{\label{sec1}Theoretical background}

In this section we present the definitions and some well-known
results related to the Laplace-Beltrami operator and periodic
manifolds. For more details on the Laplace-Beltrami operator see
e.g. \citep{Taylor}, for more details on periodic manifolds we
refer to \citep{Post_PHD}.

Let $M$ be an $n$-dimensional Riemannian manifold with the \textcolor{black}{metric}
$g$. By $g_{\a\b}$ we denote the components of $g$ in local
coordinates $(x_1,\dots,x_n)$.

As usual we denote by $L_2(M)$ the Hilbert space of square
integrable (with respect to Riemannian measure) functions on $M$.
The scalar product and norm are defined by
$$(u,v)_{L_2(M)}={\intl_{M} u\bar{v}dV},\quad  \|u\|_{L_{2}(M)}=\sqrt{(u,u)_{L_{2}(M)}}$$ where
$dV=\sqrt{\det g}d x_1{\dots}d x_n$ is the density of
\textcolor{black}{the Riemannian measure} on $M$.

By $C^\infty(M)$ (resp. $C_0^\infty(M)$) we denote the space of
smooth (resp. smooth and compactly supported) functions on $M$.

If the manifold $M$ (possibly non-compact) has an empty boundary
then we define  the \textit{Laplace-Beltrami operator} $-\Delta_M$
on $M$ in the following way. By $\bar{\eta}_M[u,v]$ we denote the
closure of the sesquilinear form $\eta_M[u,v]$ defined by the
formula:
\begin{gather}\label{eta}
\eta_M[u,v]=(\nabla u,\nabla v)_{L_2(M)}\equiv\intl_M (\nabla
u,\nabla \bar v) dV
\end{gather}
with $\mathrm{dom} (\eta_M)=C^\infty_0(M)$. Here $(\nabla u,\nabla
\bar v)$ is the scalar product of the vectors $\nabla u$ and
$\nabla \bar v$ with respect to the \textcolor{black}{metric} $g$: in local
coordinates $(\nabla u,\nabla \bar v)=\ds\suml_{\a,\b=1}^n
g^{\a\b}\ds{\partial u\over\partial x_\a}{\partial
\bar{v}\over\partial x_\b}$, where $g^{\a\b}$ are the components
of the tensor inverse to $g_{\alpha\beta}$. The form $\bar{\eta}$
is densely defined, closed and positive (by the way
$\mathrm{dom}(\bar\eta_M)=H^1(M)\equiv\left\{u\in L_2(M):\ \nabla
u\in L_2(M)\right\}$). Then there exists the unique self-adjoint
and positive operator $-\Delta_M$ associated with the form
$\bar\eta_M [u,v]$, i.e.
\begin{gather*}
(-\Delta_M u,v)_{L_2(M)}=\bar\eta_M[u,v]\text{ for all }u\in
\mathrm{dom}(\Delta_M),\ v\in \mathrm{dom}(\bar\eta_M)
\end{gather*}
For \textcolor{black}{a smooth function} $u$ the Laplace-Beltrami
operator is given in local coordinates by the formula
\begin{gather}\label{local}
-\Delta_{M}u=-\suml_{\a,\b=1}^n {1\over \sqrt{\det
g}}{\partial\over\partial x_\a}\left(g^{\a\b}\sqrt{\det g}
{\partial u\over\partial x_\beta}\right)
\end{gather}

If $M$ is a compact manifold with a piecewise smooth boundary
$\partial M$ we define the \textit{Laplace-Beltrami operator with
Neumann} (resp. \textit{Dirichlet}) \textit{boundary conditions}
$-\Delta_M^N$ (resp. $-\Delta_M^D$) as the operator associated
with the sesquilinear form  $\bar\eta^N_M$ (resp. $\bar\eta^D_M$)
which is the closure of the form $\eta^N_M$ (resp. $\eta^D_M$)
defined by \textcolor{black}{formula} (\ref{eta}) and by the
definitional domain $\mathrm{dom} (\eta_M^N)=C^\infty(M)$ (resp.
$\mathrm{dom} (\eta^D_M)=C_0^\infty({M})$).

The spectra of the operators $-\Delta_M^N$ and $-\Delta_M^D$ are
purely discrete. We denote by
$\left\{\lambda_k^N(\M)\right\}_{k\in\mathbb{N}}$ (resp.
$\left\{\lambda_k^D(\M)\right\}_{k\in\mathbb{N}}$)  the sequence
of eigenvalues of $-\Delta_M^N$ (resp. $-\Delta_M^D$) written in
the increasing order and repeated according to their
multiplicity.\medskip

Now we present \textcolor{black}{the concept} of periodic Riemannian
manifolds.

We say that \textit{the group $\Gamma$ acts on the manifold $M$}
if there is a map $\Gamma\times M\to M$ (denoted $(\gamma,
x)\mapsto \gamma\cdot x$) such that
$\forall\gamma_1,\gamma_2\in\Gamma$, $\forall x\in M$ one has
$(\gamma_1*\gamma_2)\cdot x=\gamma_1\cdot(\gamma_2\cdot x)$, where
$*$ is the group operation, and $\forall x\in M$ one has
${id}\cdot x=x$, where ${id}$ is the identity element of $\Gamma$.

The Riemannian manifold $M$ is called \textit{periodic} (or more
precisely \textit{$\Gamma$-periodic}) if a discrete finitely
generated abelian group $\Gamma$ acts on $M$, moreover
\begin{itemize}
\item $\Gamma$ acts \textit{isometrically} on $M$, i.e.
$\forall\gamma\in\Gamma$: $\gamma\cdot$ is the isometrical map,
\item $\Gamma$ acts \textit{properly discontinuously} on $M$, i.e.
for each $x\in M$ there exists a neighbourhood $U_x$ such that the
sets $\gamma\cdot U_x$ $(\gamma\in\Gamma)$ are pairwise
disjoint,\item $\Gamma$ acts \textit{co-compactly} on $M$, i.e. the
quotient space $M/\Gamma$ is compact.
\end{itemize}

A compact subset $\mathbf{M}\subset M$ is called a \textit{period
sell} if $\cupl_{\gamma\in\Gamma}\gamma\cdot \mathbf{M}=M$ and
$\mathbf{M}$ is a closure of an open connected domain $\mathbf{D}$
such that $\forall\gamma\in\Gamma,\ \gamma\not={id}:
\mathbf{D}\cap \gamma\cdot \mathbf{D}=\varnothing$.

For convenience \textcolor{black}{throughout our work}  we will use
the same notation $\gamma$ for the element $\gamma\in\Gamma$ and
the corresponding map $\gamma\cdot:M\to M$.

By $\hat\Gamma$ we denote the dual group of $\Gamma$, i.e. the
group of homomorphism from $\Gamma$ into $\mathbb{S}^1$. We remark
that if $\Gamma$ is isomorphic to $\mathbb{Z}^n$ (as for the
manifold $M\e$, which will be considered in the next section) then
$\hat\Gamma$ is isomorphic to the $n$-dimensional torus
$\mathbb{T}^n=\left\{\theta=(\theta_1,\dots,\theta_r)\in
\mathbb{C}^n:\ \forall \a\ |\theta_\a|=1\right\}$.

Let $\theta\in\hat\Gamma$. We define \textit{the Laplace-Beltrami
operator with $\theta$-periodic boundary conditions}
$-\Delta^{\theta}_{\M}$ in the following way. By
$C^\infty_\theta(\mathbf{M})$ we denote the space of functions
$u\in C^\infty(\mathbf{M})$ satisfying
\begin{gather*}
u(\gamma x)=\overline{\theta(\gamma)}u(x)
\end{gather*}
for each $x\in\partial\mathbf{M}$ and for each $\gamma\in\Gamma$
such that $\gamma x\in \partial\mathbf{M}$. Then we define the
operator $-\Delta_{\M}^\theta$ as the operator associated with the
form $\bar\eta_{\M}^\theta$ which is the closure of the form
$\eta^\theta_{\M}$  defined by \textcolor{black}{formula}
(\ref{eta}) (with $\mathbf{M}$ instead of $M$) and by the
definitional domain $\mathrm{dom}
(\eta_{\M}^\theta)=C^\infty_\theta(\mathbf{M})$.

The operator $-\Delta_{\M}^\theta$ has purely discrete spectrum.
We denote by
$\left\{\lambda_k^\theta(\M)\right\}_{k\in\mathbb{N}}$ the
sequence of eigenvalues of $-\Delta_{\M}^\theta$ written in the
increasing order and repeated according to their multiplicity.

For any $\theta\in\hat\Gamma$ the following inequality holds:
\begin{gather}\label{enclosure}
\lambda_k^N(\M)\leq\lambda_k^\theta(\M)\leq\lambda_k^D(\M)
\end{gather}

It turns out that analysis of the spectrum $\sigma(-\Delta_M)$ of
the operator $-\Delta_M$ on the periodic manifold $M$ can be
reduced to analysis of the spectra $\sigma(-\Delta^\theta_{\M})$
of the operators $\sigma(-\Delta^\theta_{\M})$, $\theta\in
\hat\Gamma$. Namely one has the following fundamental
result.\smallskip

\noindent\textbf{Theorem.}\textit{ Let $M$ be $\Gamma$-periodic
manifold with a period cell $\mathbf{M}$. Then
\begin{gather}\label{fund}
\sigma(-\Delta_M)=\cupl_{k\in \mathbb{N}}\mathcal{J}_k(M)
\end{gather}
where $\mathcal{J}_k({M})=\left\{\lambda_k^\theta(\mathbf{M}):\
\theta\in\hat\Gamma\right\}$, $k\in\mathbb{N}$ are compact
intervals. }

\section{\label{sec2}Construction of the manifold}
In this section we construct the manifold $M\e$ and describe the
behaviour of the spectrum $\sigma(-\Delta_{M\e})$ of the
Laplace-Beltrami operator $-\Delta_{M\e}$ as $\eps\to 0$.

Let $\left\{D_{ij}\e:\ i\in \mathbb{Z}^n,\ j=1,...,m\right\}$ be
the system of pairwise disjoint balls in $\mathbb{R}^n$ ($n\geq
2$) depending on small parameter $\eps>0$. We suppose that:
\begin{itemize}
\item[1)] the balls $D_{0j}\e$, $j=1,...,m$ belong to the cube
$\square_0\e=\left\{x\in\mathbb{R}^n:\ 0\leq x_\a\leq\eps,\
\forall\a\right\}$;

\item[2)] $\forall j=1,...,m$:
$\kappa\eps\leq\mathrm{dist}\left(D_{0j}\e,\
\partial\square_0\e\cup\left(\cupl_{l\not=j}D_l\e\right)\right)$,
where the constant $\kappa>0$ is independent of $\eps$;

\item[3)] $\forall i\in \mathbb{Z}^n,\ \forall j=1,...,m$:
$D_{ij}\e=D_{0j}\e+\eps i$.

\end{itemize}

By $x_{ij}\e$ we denote the centre of $D_{ij}\e$, by $d_j\e$ we
denote the radius of $D_{ij}\e$ (the third condition above implies
that the radius of $D_{ij}\e$ depends only on the index $j$).

We denote by $B_{ij}\e$ the truncated $n$-dimensional sphere (we
call it "bubble") of the radius $b_{j}\e$:
\begin{gather*}
B_{ij}\e=\left\{
(\theta_1,\theta_2,...,\theta_n):\theta_1\in[0,2\pi),\
\theta_k\in[0,\pi),\ k=2,...,n-1,\
\theta_{n}\in[\Theta_j\e,\pi]\right\}
\end{gather*}
Here $\Theta_{j}\e=\ds\arcsin\left({d_{j}\e\over b_{j}\e}\right)$,
where $b_j\e$ ($j=1,\dots,m$) are positive numbers satisfying
$b_j\e>d_j\e$.

Let us introduce in $\Omega\e$ the spherical coordinates
$(\theta_1,\dots,\theta_n,r)$ with the origin at $x_{ij}\e$. Here
$r$ is the distance to $x_{ij}\e$. Identifying the points
$\big(\theta_1,\dots,\theta_{n-1},d_{j}\e)\in\partial D_{ij}\e$
and $\big(\theta_1,\dots,\theta_{n-1},\Theta_j\e\big)\in\partial
B_{ij}\e$ we glue the bubbles $B_{ij}\e$ to the perforated domain
$\Omega\e$ and obtain an $n$-dimensional manifold $M\e$:
\begin{gather}\label{manifold}
M\e={\Omega\e}\cup\left(\cupl_{i\in
\mathbb{R}^n}\cupl_{j=1}^m{B_{ij}\e}\right)
\end{gather}
The manifold $M\e$ is presented on the Figure \ref{fig1}. By
$\tilde x$ we denote the points of $M\e$. If the point $\tilde x$
belongs to $\Omega\e$ sometimes we will write $x$ instead of
$\tilde x$ having in mind a corresponding point in $\mathbb{R}^n$.

Clearly $M\e$ can be covered by a system of charts and suitable
local coordinates $(x_1,\dots,x_n)\mapsto \tilde x\in M\e$ can be
introduced. In particular in a small neighbourhood of $\partial
B_{ij}\e$ we introduce them in the following way (below by
$U_{ij}\e$ we denote this neighbourhood):
\begin{gather}\label{coord}
\begin{matrix}
x_k=\theta_k,\ k=1,\dots,n-1,\\
x_n=\begin{cases}
r-d_{j}\e,&\tilde x=(\theta_1,\dots,\theta_{n-1},r)\in\Omega\e\cap U_{ij}\e,\\
-b_{j}\e\left(\theta_n-\Theta_{j}\e\right),&\tilde
x=(\theta_1,\dots,\theta_{n-1},\theta_n)\in B_{ij}\e\cap U_{ij}\e.
\end{cases}
\end{matrix}
\end{gather}
(that is $\partial B_{ij}\e=\left\{(x_1,\dots,x_n):\
x_n=0)\right\}$).

We \textcolor{black}{equip $M\e$ with} the Riemannian
\textcolor{black}{metric} $g\e$ that coincides with the flat
Euclidean \textcolor{black}{metric} on $\Omega\e$ and coincides with
the spherical \textcolor{black}{metric} on the bubbles $B_{ij}\e$.
This last means that in the spherical coordinates
$(\theta_1,\dots,\theta_n)$ the components $g_{\a\b}\e$ of the
\textcolor{black}{metric} $g\e$ have the form
\begin{gather*}
g_{\a\b}\e=\delta_{\a\b}\left(b_{j}\e\right)^2
\prod\limits_{k=\a+1}^{n}\sin^2 \theta_k,\quad
\a,\b=\overline{1,n}
\end{gather*}
(for $\a=n$ we set $\prod\limits_{k=\a+1}^{n}\sin^2 \theta_k$:=1).
Here $\delta_{\a\b}$ is the Kronecker delta.

The \textcolor{black}{metric} $g\e$ is continuous and piecewise
smooth: in the coordinates $(x_1,\dots,x_n)$, which are introduced
above in the neighbourhood of $\partial B_{ij}\e$ by
\textcolor{black}{formulae} (\ref{coord}), the components
$g_{\a\b}\e=g_{\a\b}\e(x_1,\dots,x_n)$ of the
\textcolor{black}{metric} $g\e$ have the form:
\begin{gather}\label{metrics}
g_{\a\b}\e=\begin{cases}g_{+_{\a\b}}^{\eps},&x_n\geq
0,\\g_{-_{\a\b}}\e,& x_n< 0,
\end{cases}\ \a,\b=\overline{1,n-1},\quad
g_{n\beta}=\delta_{n\beta}\\\label{metrics+} \text{where}\quad
g_{+_{\a\b}}^{\eps}=\delta_{\a\b}\left(x_n+d_{j}\e\right)^2
\prod\limits_{k=\a+1}^{n-1}\sin^2 \theta_k,\quad
g_{-_{\a\b}}\e=\delta_{\a\b}\left(b_{j}\e\right)^2\sin^2\left(\ds{|x_n|\over
b_j\e}+\Theta_j\e\right) \prod\limits_{k=\a+1}^{n-1}\sin^2
\theta_k
\end{gather}
It is clear that as $x_n=0$ (i.e. on $\partial B_{ij}\e$) the
coefficients $g_{\a\b}\e$ lose smoothness.

Remark that $g\e$ can be approximated by a smooth \textcolor{black}{metric}
$g^{\eps\rho}$ that differs from $g\e$ only in a small
$\rho$-neighbourhood of $\partial B_{ij}\e$, moreover when
$\rho=\rho(\eps)$ is sufficiently small then the spectra
$\sigma(-\Delta_{M\e})$ and $\sigma(-\Delta_{(M\e,g^{\eps\rho})})$
have the same limit as $\eps\to 0$ (for more precise statement see
Remark \ref{rem22}). However in order to omit cumbersome
calculations further we will work with the \textcolor{black}{metric} $g\e$.

\begin{remark} It is easy to see that the manifold $M\e$ can be immersed into the space
$\mathbb{R}^{n+1}$ via the following map $\widehat F\e:
M\e\to\widehat{M}\e\subset \mathbb{R}^{n+1}$ (below
$x\in\mathbb{R}^n,\ z\in\mathbb{R},\ (x,z)\in\mathbb{R}^{n+1})$:
\begin{itemize}
\item[-] if $\tilde x=x\in\Omega\e$ then $\widehat F\e(\tilde
x)=(x,0)$, \item[-] if $\tilde x=(\theta_1,\dots,\theta_n)\in
B_{ij}\e$ then $\widehat F\e(\tilde x)=(x_1,\dots, x_n,z)$, where
\end{itemize}
\quad$
x_1=\big(x_{ij}\e\big)_1+b_{j}\e\prod\limits_{l=1}^n\sin\theta_l,\quad
x_k=\big(x_{ij}\e\big)_k+b_{j}\e
\cos\theta_{k-1}\prod\limits_{l=k}^n\sin\theta_l\
(k=\overline{2,n}),\quad z=b_{j}\e(\cos\Theta_{j}\e-\cos\theta_n)
$\medskip\\ Note: one should not confuse $(x_1,\dots,x_n)$ with
the local coordinates introduced above in a neighbourhood of
$\partial B_{ij}\e$.

\textcolor{black}{Thus,} $\widehat{F}\e$ maps $B_{ij}\e$ onto the
surface $\widehat{B}_{ij}\e$ which is obtained by removing from
the sphere $\widehat{\mathcal
B}_{ij}\e=\left\{(x,z)\in\mathbb{R}^{n+1}:\
|x-x_{ij}\e|^2+(z-b_j\e\cos\Theta_{j}\e)^2=(b_j\e)^2\right\}$ the
segment $\big\{(x,z)\in\widehat{\mathcal B}_{ij}\e:\ z< 0\big\}$.

The map $\widehat F\e$ is a local {homeomorphism}, i.e. for any
$\tilde x\in M\e$ there is a neighbourhood $U(\tilde x)\subset
M\e$ such that $\widehat F\e|_{U(\tilde x)}$ is a homeomorphism
(and even diffeomorphism if $\tilde x\notin\cupl_{i,j}\partial
B_{ij}\e$). If the surfaces $\widehat{B}_{ij}\e$
($i\in\mathbb{Z}^n$, $j=1,\dots,m$) are pairwise disjoint (e.g. if
$b_j\e<d_j\e+\kappa\eps/2$) then $\widehat F\e$ is a global
homeomorphism. Furthermore $\widehat{F}\e$ is an isometric map: if
$\hat g\e$ is a \textcolor{black}{metric} on $\widehat M\e$ which is generated by
the Euclidean \textcolor{black}{metric} in $\mathbb{R}^{n+1}$ then $g\e$ coincides
with the pull-back $(\widehat F\e)^*\hat{g}\e$.
\end{remark}

Let the group $\Gamma\e\cong\mathbb{Z}^n$ act on $M\e$ by the
following rule (below by $\gamma_k\e$, $k\in\mathbb{Z}^n$ we
denote the elements of $\Gamma\e$):
\begin{itemize}
\item[-] if $\tilde x=x\in\Omega\e$ then $\gamma_k\e$ maps $\tilde
x$ into the point $\gamma_k\tilde x=x+k\eps\in\Omega\e$,

\item[-] if $\tilde x=(\theta_1,\dots,\theta_n)\in B_{ij}\e$ then
$\gamma_k\e$ maps $\tilde x$ into the point $\gamma_k\e\tilde x\in
B_{i+k,j}$ with the same angle coordinates
$(\theta_1,\dots,\theta_n)$.
\end{itemize}
Obviously $M\e$ is $\Gamma\e$-periodic Riemannian manifold. For an
arbitrary $i\in\mathbb{Z}^n$ the set
\begin{gather}\label{cell}
\M_i\e=F_i\e\cup \left(\cupl_{j=1}^m B_{ij}\e\right)\text{,\quad
where } F_i\e=\left\{\tilde x\in\Omega\e:\ x-\eps
i\in\square_0\e\setminus\left(\cupl_{j=1}^m
D_{0j}\e\right)\right\}
\end{gather}
is a period cell.

We assume that the radii of the holes and bubbles are the
following:
\begin{gather}\label{d_size}
d_j\e=\begin{cases}\ds d_j\eps^{{n\over n-2}},&
n>2\\\ds\exp\left(-{1\over
d_j\eps^2}\right),&n=2\end{cases}\\\label{b_size} b_j\e=b_j\eps
\end{gather}
where $d_j,\ b_j$ ($j=1,\dots,m$) are some positive constants (we
choose them later in Section \ref{sec4}).

\textcolor{black}{We will use the following notations:}
\begin{gather*}
R_{ij}\e=\left\{\tilde x\in\Omega\e:\ d_{j}\e\leq|x-x_{ij}\e|<d_{j}\e+{\kappa\eps\over 2}\right\},\\
G_{ij}\e={R_{ij}\e\cup B_{ij}\e},\\
S_{ij}^{\eps}=\left\{\tilde x\in\Omega\e:\
|x-x_{ij}\e|=d_{j}\e+{\kappa\eps\over
2}\right\}\equiv\partial G_{ij}\e,\\
\omega_n\text{ is the volume of }n\text{-dimensional unit sphere}
\end{gather*}

According to the notations introduced above in Section \ref{sec1}
we denote by
$\left\{\lambda_{k}^D(G_{ij}\e)\right\}_{k\in\mathbb{N}}$ the
sequence of the eigenvalues of the operator $-\Delta_{G_{ij}\e}^D$
which is the Laplace-Beltrami operator in $G_{ij}\e$ with
Dirichlet boundary conditions on $S_{ij}\e$. It is clear that
$\left\{\lambda_{k}^D(G_{ij}\e)\right\}_{k\in\mathbb{N}}$ depends
only on the index $j$.

One can prove (see Lemma \ref{lm0} below) that
$$\forall j=1,\dots,m:
\quad \liml_{\eps\to 0}\lambda_1(G_{ij}\e)=\sigma_j$$ where
\begin{gather}\label{sigma}\sigma_j=
\begin{cases}
\ds{d_j\over 4b_j^2},&n=2\\\ds {n-2\over 2}\cdot
{d_j^{n-2}\omega_{n-1}\over b_j^{n}\omega_{n}},&n>2
\end{cases}
\end{gather}
Note that in spite of the fact that the diameter of $G_{ij}\e$
converges to zero as $\eps\to 0$, $\lambda_{1}(G_{ij}\e)$ does not
blow up as $\eps\to 0$. This is due to a weak connection between
$B_{ij}\e$ and $R_{ij}\e$.

We assume that the coefficients $d_j$ and $b_j$ are such that
$\sigma_i\not=\sigma_j$ if $i\not=j$. For definiteness we suppose
that $\sigma_j<\sigma_{j+1}$, $j=1,\dots,n-1$.

We introduce the Hilbert space
$$H=L_2(\mathbb{R}^n)\underset{j=\overline{1,m}}\oplus
L_2(\mathbb{R}^n,\rho_j dx)$$ where by $dx$ we denote the density
of the Lebesgue measure, the constant weights $\rho_j$,
$j=1,\dots,m$ are defined by the formula
\begin{gather}\label{rho}
\rho_j=(b_j)^n\omega_{n}
\end{gather}
Since $\ds\liml_{\eps\to 0}\left({d_{j}\e/ b_{j}\e}\right)=
0$\textcolor{black}{, then} $\rho_j=\liml_{\eps\to
0}\eps^{-n}|B_{ij}\e|$ (here by $|\cdot|$ we denote the Riemannian
volume).

And\textcolor{black}{, finally,} let us consider the following
equation (with unknown $\lambda\in\mathbb{R}$):
\begin{gather}\label{mu_eq}
\mathcal{F}(\lambda)\equiv
1+\suml_{j=1}^m{\sigma_j\rho_j\over\sigma_j-\lambda}=0
\end{gather}
It is easy to obtain (see the proof of Theorem \ref{th2}) that
this equation has exactly $m$ roots $\mu_j$ ($j=1,\dots,m$),
moreover one can renumber them in such a way that
$$
\sigma_j<\mu_j<\sigma_{j+1},\ j=\overline{1,m-1},\quad
\sigma_m<\mu_m<\infty
$$
By the way if $m=1$ then $\mu_1=\sigma_1+\sigma_1\rho_1$ (cf.
Remark \ref{rem00}).

Now we are able to formulate the theorem describing the behaviour
of $\sigma(-\Delta_{M\e})$.
%-------------------------------------------------------------------
%--------------THEOREM 2 (HOMOGENIZATION)---------------------------
%-------------------------------------------------------------------
\begin{theorem}\label{th2}
The spectrum $\sigma(-\Delta_{M\e})$ of the operator
$-\Delta_{M\e}$ has the following structure \textcolor{black}{when
$\eps$ is small enough} (i.e. \textcolor{black}{when
$\eps<\eps_0$}):
\begin{gather}\label{Haus0}
\sigma(-\Delta_{M\e})=[0,\infty)\setminus\left[\left(\cupl_{j=1}^m(\sigma_j\e,\mu_j\e)\right)\cup
\mathcal{J}\e\right]
\end{gather}
Here $\mathcal{J}\e$ is a union of some open finite intervals
(possibly $\mathcal{J}\e=\varnothing$) and
\begin{gather*}
0<\sigma\e_1,\quad \sigma\e_{j}<\mu\e_{j}<\sigma\e_{j+1},\
j=\overline{1,m-1},\quad\sigma_m\e<\mu\e_m<\inf\mathcal{J}\e
\end{gather*}
Moreover
\begin{gather}\label{Haus1}
\forall j=1,\dots,m:\quad \lim_{\eps\to
0}\sigma_j\e=\sigma_j,\quad \lim_{\eps\to
0}\mu_j\e=\mu_j\\\label{Haus2} \liml_{\eps\to
0}\inf\mathcal{J}\e=\infty
\end{gather}

The set
$[0,\infty)\setminus\left(\cupl_{j=1}^m(\sigma_j,\mu_j)\right)$
coincides with the spectrum $\sigma(\mathcal{A})$ of the
self-adjoint operator $\mathcal{A}$ which acts in $H$ and is
defined by the formulae
\begin{gather}
\label{A} \mathcal{A} U= \left(\begin{matrix} -\Delta_{\mathbb{R}^n} u+\ds\suml_{j=1}^m\sigma_j\rho_j(u-u_j)\\
\sigma_1(u_1-u)\\\sigma_2(u_2-u)\\\dots\\\sigma_m(u_m-u)
\end{matrix}\right),\ U=\left(\begin{matrix}u\\u_1\\u_2\\\dots\\u_m\end{matrix}\right)\in
\mathrm{dom}(\mathcal{A})\\\label{A_dom}\mathrm{dom}(\mathcal{A})=\mathrm{dom}({\Delta_{\mathbb{R}^n}})\underset{j=\overline{1,m}}\oplus
L_2(\mathbb{R}^n,\rho_j dx)
\end{gather}
\end{theorem}
%-------------------------------------------------------------------
%-------------------------------------------------------------------
%-------------------------------------------------------------------

We prove this theorem in the next section. In the last section we
present the formulae for $d_j$, $b_j$ which will ensure the
fulfilment of the equalities (\ref{ab}).

\section{\label{sec3}Proof of Theorem \ref{th2}}

\textcolor{black}{Before we prove the result in full detail we will
sketch the main ideas of the proof.}

At first (Subsection \ref{ss21}) we prove the equality
\begin{gather}\label{struct}
\sigma(\mathcal{A})=[0,\infty)\setminus\left(\cupl_{j=1}^m(\sigma_j,\mu_j)\right)
\end{gather}

\textcolor{black}{In the main part of the proof (Subsections
\ref{ss22}-\ref{ss23})} we show that
\\
\textit{for an arbitrary $L>0$,
$L\notin\cupl_{j=1}^m\left\{\mu_j\right\}$ the set
$\sigma(-\Delta_{M\e})\cap[0,L]$ converges in the Hausdorff sense
to the set $\sigma(\mathcal{A})\cap[0,L]$ as $\eps\to 0$.}\medskip

Let us recall the definition of \textcolor{black}{Hausdorff
convergence}.

\begin{definition} The set $\mathcal{B}\e\subset\mathbb{R}$ converges in the
Hausdorff sense to the set $\mathcal{B}\subset\mathbb{R}$ as
$\eps\to 0$ if the following conditions (A) and (B) hold:
\begin{gather}\tag{A}\label{ah}
\text{if }\lambda\e\in\mathcal{B}\e\text{ and }\liml_{\eps\to
0}\lambda\e=\lambda\text{ then }\lambda\in
\mathcal{B}\\\tag{B}\label{bh} \text{for any }\lambda\in
\mathcal{B}\text{ there exists }\lambda\e\in\mathcal{B}\e\text{
such that }\liml_{\eps\to 0}\lambda\e=\lambda
\end{gather}
\end{definition}

\textcolor{black}{Property (\ref{ah})} is verified in Subsection
\ref{ss22}, \textcolor{black}{property (\ref{bh})} is verified in
Subsection \ref{ss23}.

\textcolor{black}{In the last part of the proof (Subsection
\ref{ss24})} we show that within an arbitrary finite interval
$[0,L]$ \textcolor{black}{the spectrum $\sigma(-\Delta_{M\e})$ has
at most $m$ gaps} \textcolor{black}{when $\eps$ is small enough}
(i.e. \textcolor{black}{when $\eps<\eps_0$}). This fact and the
Hausdorff convergence of $\sigma(-\Delta_{M\e})\cap[0,L]$ to
$\sigma(\mathcal{A})\cap[0,L]=[0,L]\setminus\left(\cupl_{j=1}^m(\sigma_j,\mu_j)\right)$
imply the properties (\ref{Haus0})-(\ref{Haus2}). Indeed one can
easily prove the following simple proposition.
\begin{proposition}\label{prop1}
Let
$\mathcal{B}\e=[0,L]\setminus\left(\cupl_{j=1}^{m\e}(\a_j\e,\b_j\e)\right)$,
$\mathcal{B}=[0,L]\setminus\left(\cupl_{j=1}^{m}(\a_j,\b_j)\right)$,
where $L<\infty$ and
\begin{gather*}
0\leq\a_1\e,\quad \a\e_{j}<\b\e_{j}\leq \a\e_{j+1},\
j=\overline{1,m\e-1},\quad \a\e_{m\e}\leq L\\ 0<\a_1,\quad
\a_{j}<\b_{j}<\a_{j+1},\
j=\overline{1,m-1},\quad \a_{m}<L\\
m\e\leq m
\end{gather*}
Suppose that the set $\mathcal{B}\e$ converges to the set
$\mathcal{B}$ in the Hausdorff sense as $\eps\to 0$.

Then $m\e=m$ when $\eps$ becomes small (i.e. when $\eps$ is less
than some $\eps_0$) and
$$\forall j=1,\dots,m:\quad \liml_{\eps\to 0}\a_j\e=\a_j,\quad \liml_{\eps\to 0}\b_j\e=\b_j$$
\end{proposition}

\subsection{\label{ss21}Structure of $\sigma(\mathcal{A})$}

Let $\lambda\in
\mathbb{C}\setminus\cupl_{j=1}^m\left\{\sigma_j\right\}$. Let
$F=\left(\begin{matrix}f\\f_1\\\dots\\f_m\end{matrix}\right)\in
\mathrm{im}(\mathcal{A}-\lambda \mathrm{I})$, i.e. there is
$U=\left(\begin{matrix}u\\u_1\\\dots\\u_m\end{matrix}\right)\in
\mathrm{dom}(\mathcal{A})$ satisfying $\mathcal{A}U-\lambda U=F$.
Then $u_j=\ds{\sigma_j u+f_j\over\sigma_j-\lambda}$ and
\begin{gather}\label{pencil}
-\Delta_{\mathbb{R}^n} u-\lambda\mathcal{F}(\lambda)
u=f+\suml_{j=1}^m{\sigma_j\rho_j f_j\over\sigma_i-\lambda}
\end{gather}
where $\mathcal{F}(\lambda)$ is defined by (\ref{mu_eq}).
\textcolor{black}{Equality} (\ref{pencil}) implies that
\begin{gather}\label{arg}
\lambda\in
\sigma(\mathcal{A})\setminus\cupl_{j=1}^m\left\{\sigma_j\right\} \
\Longleftrightarrow\ \lambda\mathcal{F}(\lambda)\in
\sigma(-\Delta_{\mathbb{R}^n})=[0,\infty)
\end{gather}

\begin{figure}[b]
\begin{center}
\scalebox{0.45}[0.45]{\includegraphics{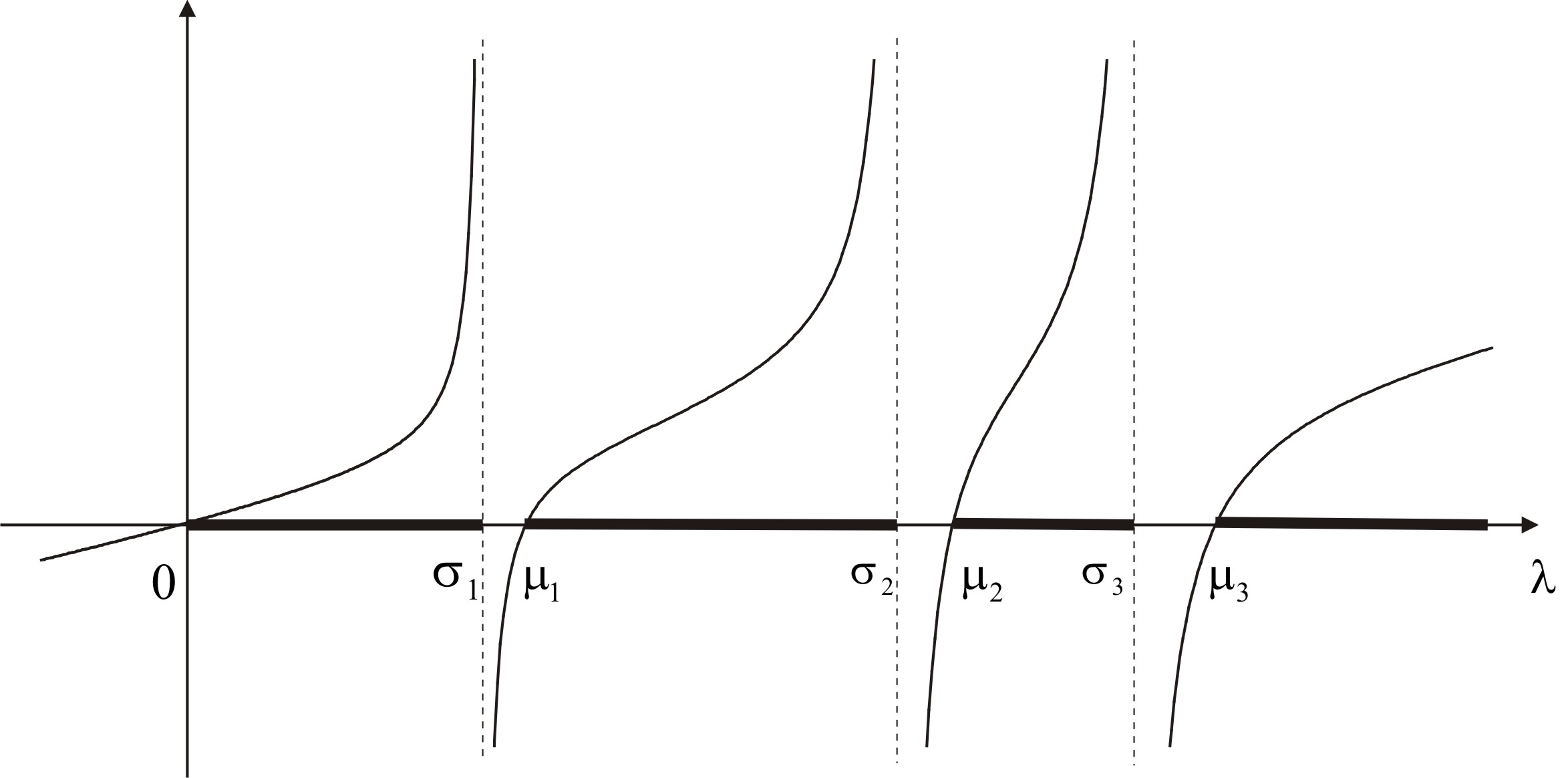}}
\caption{\label{fig2}The \textcolor{black}{graph} of the function
$\lambda\mathcal{F}(\lambda)$ (for $m=3$). The bold intervals are
the components of $\sigma(\mathcal{A})$}
\end{center}
\end{figure}

At first we study the function $\lambda\mathcal{F}(\lambda)$ on
the real line. It is easy to see that
$\lambda\mathcal{F}(\lambda)$ is a strictly increasing function on
the intervals $(-\infty,\sigma_1)$, $(\sigma_m,\infty)$,
$(\sigma_j,\sigma_{j+1})$, $j=1,\dots,m-1$,
$\liml_{\lambda\to\pm\infty}\lambda\mathcal{F}(\lambda)=\pm\infty$,
$\liml_{\lambda\to\sigma_j\pm
0}\lambda\mathcal{F}(\lambda)=\mp\infty$, furthermore there are
the points $\mu_j$, $j=1,\dots,m$, such that
\begin{gather*}
\mathcal{F}(\mu_j)=0,\ j=1,\dots,m-1\\
\sigma_j<\mu_j<\sigma_{j+1},\ j=1,\dots,m\quad
\sigma_m<\mu_m<\infty\\
\left\{\lambda\in\mathbb{R}\setminus\cupl_{j=1}^m\left\{\sigma_j\right\}:\
\lambda\mathcal{F}(\lambda)\geq 0\right\}=[0,\sigma_1)\cup
\left(\cupl_{j=1}^{m-1}[\mu_j,\sigma_{j+1})\right)\cup[\mu_m,\infty)
\end{gather*}

Let us consider the equation $\lambda\mathcal{F}(\lambda)=a$,
where $a\in [0,\infty)$. One the one hand it is equivalent to the
equation $\left(\prod\limits_{j=1}^m(\sigma_j-\lambda)\right)^{-1}
P_{m+1}(\lambda)=0$, where $P_{m+1}$ is a polynomial
\textcolor{black}{of the degree} $m+1$, and\textcolor{black}{,
therefore,} in $\mathbb{C}$ this equation has \textcolor{black}{at
most $m+1$ roots}. On the other hand it is easy to see that on
$[0,\infty)$ the equation $\lambda\mathcal{F}(\lambda)=a$ has
$m+1$ roots  (if $a=0$ then these roots are
$0,\mu_1,\dots,\mu_m$). Hence we obtain that \textcolor{black}{the
set $\{\lambda\in \mathbb{C}:\lambda\mathcal{F}(\lambda)\geq 0\}$
belongs} to $[0,\infty)$.

The \textcolor{black}{graph} of the function
$\lambda\mathcal{F}(\lambda)$, $\lambda\in \mathbb{R}$ is
presented on the Figure \ref{fig2}.

\textcolor{black}{Thus,} we conclude that $\lambda\in
\sigma(\mathcal{A})\setminus\cupl_{j=1}^m\left\{\sigma_j\right\}$
iff $\lambda\in
[0,\sigma_1)\cup\left(\cupl_{j=1}^{m-1}[\mu_j,\sigma_{j+1})\right)\cup[\mu_m,\infty)$.
Since the spectrum $\sigma(\mathcal{A})$ is a closed set
\textcolor{black}{, then the points $\sigma_j$ ($j=1,\dots,m$)} also
belong to $\sigma(\mathcal{A})$. \textcolor{black}{Equality}
(\ref{struct}) is proved.

\subsection{\label{ss22}Property (\ref{ah}) of Hausdorff convergence}

We present \textcolor{black}{the proof} for the case $n\geq 3$ only.
For the case $n=2$ the proof is repeated word-by-word with small
modifications \textcolor{black}{in some  estimates}.

Let $\lambda\e\in\sigma(-\Delta_{M\e})\cap[0,L]$ and
$\liml_{\eps\to 0}\lambda\e=\lambda$. Obviously $\lambda\in[0,L]$,
\textcolor{black}{thus,} we have to prove that $\lambda$ belongs to
$\sigma(\mathcal{A})$. If
$\lambda\in\cupl_{j=1}^m\left\{\sigma_j\right\}$ then
\textcolor{black}{this statement follows from (\ref{struct})}.
\textcolor{black}{Therefore,} \textcolor{black}{we can focus on the
case} $\lambda\notin\cupl_{j=1}^m\left\{\sigma_j\right\}$.

Let us consider \textcolor{black}{the sequence} $\eps_N\subset\eps$,
where $\eps_N={1\over N}$, $N=1,2,3$\dots\ For convenience we
preserve the same notation $\eps$ having in mind the sequence
$\eps_N$.

We introduce the following cubes in $\mathbb{R}^n$:
\begin{gather*}
\square=\left\{x\in\mathbb{R}^n:\
0\leq x_\a\leq {1},\ \forall\a\right\}\\
\square_i\e=\left\{x\in\mathbb{R}^n:\ \eps i_\a\leq x_\a\leq
\eps(i_\a+1),\ \forall\a\right\},\
i=(i_1,\dots,i_n)\in\mathbb{Z}^n
\end{gather*}
Since $\eps^{-1}\in\mathbb{N}$\textcolor{black}{, then}
$\square=\cupl_{i\in\I}\square_i\e$, where
$$\I=\left\{i\in \mathbb{Z}^n:\ 0\leq i_\a\leq (\eps^{-1}-1),
\forall\alpha \right\}$$ Also we introduce the following set in
$M\e$:
\begin{gather*}
\M\e=\cupl_{i\in \I}\M_i\e
\end{gather*}
where $\M_i\e$ is defined by \textcolor{black}{formulae}
(\ref{cell}).

In Section \ref{sec2} \textcolor{black}{we concluded} that $M\e$ is
$\Gamma\e$-periodic manifold, the set $\M_i\e$ is a corresponding
periodic cell. On the other hand since
$\eps^{-1}\in\mathbb{N}$\textcolor{black}{, then $M\e$ is also
$\Gamma$-periodic manifold on which the group
$\Gamma\cong\mathbb{Z}^n$ acts} by the following rule (below by
$\gamma_k$, $k\in\mathbb{Z}^n$ we denote the elements of
$\Gamma$):
\begin{itemize}
\item[-] if $\tilde x=x\in\Omega\e$ then $\gamma_k$ maps $\tilde
x$ into the point $\gamma_k\tilde x=x+k\in\Omega\e$,

\item[-] if $\tilde x=(\theta_1,\dots,\theta_n)\in B_{ij}\e$ then
$\gamma_k$ maps $\tilde x$ into the point $\gamma_k\tilde x\in
B_{i+k\eps^{-1},j}$ with the same angle coordinates
$(\theta_1,\dots,\theta_n)$.
\end{itemize}
The set $\M\e$ is a period cell. The boundary of $\M\e$ is
independent of $\eps$: $\partial\M\e=\left\{\tilde x\in\Omega\e:\
x\in\partial\square\right\}$.

Roughly speaking if $\eps^{-1}\in\mathbb{N}$ then $M\e$ is not
only "$\eps$-periodic" manifold but also "$1$-periodic" manifold.
\textcolor{black}{To prove} \textcolor{black}{property (\ref{ah})} of
the Hausdorff convergence it is more convenient to look at $M\e$
as $\Gamma$-periodic manifold (and to work with period cell
$\M\e$) since in this case we are able to utilize some ideas and
methods developed in \citep{Khrus1,Khrus2,Khrab3,Khrab5,Khrab6}.

By $\mathbf{M}_\a$ ($\a=1,\dots,2n$) we denote the components of
$\partial\M\e$:
\begin{gather*}\begin{matrix}
\mathbf{M}_\a=\left\{\tilde x\in\Omega\e:\ x_\a=0\text{ and }
0\leq x_\b\leq 1,
\forall\b\not=\alpha\right\}&\text{ if }\alpha=1,\dots, n\\
\mathbf{M}_\a= \left\{\tilde x\in\Omega\e:\ x_{\a-n}=1\text{ and }
0\leq x_\b\leq 1, \forall\b\not=\alpha-n\right\}&\text{ if
}\alpha=n+1,\dots, 2n
\end{matrix}
\end{gather*}
The faces $\mathbf{M}_\a$ and $\mathbf{M}_{\a+n}$ ($\a=1,\dots,n$)
are parallel to each other and
\begin{gather}\label{faces}
\gamma_{e_\a}\M_\a=\M_{\a+n},\ \a=1,\dots,n,\text{ where }
e_\a=\underset{^{\overset{\qquad\quad\uparrow}{\qquad\quad
\a\text{-th place}}\qquad }}{(0,0,\dots,1,\dots,0)}
\end{gather}
Also \textcolor{black}{we denote} by $\mathbf{M}_\a$ the
corresponding faces of $\partial\square$.

Since $\lambda\e\in\sigma(-\Delta_{M\e})$\textcolor{black}{, then
}there exists $\theta\e\in\hat\Gamma$ such that
$\lambda\e\in\sigma(-\Delta_{\M\e}^{\theta\e})$. Since $\Gamma$ is
isomorphic to $\mathbb{Z}^n$\textcolor{black}{, then} the dual group
$\hat\Gamma$ is isomorphic to
$\mathbb{T}^n=\left\{\theta=(\theta_1,\dots,\theta_r)\in
\mathbb{C}^r:\ \forall \a\ |\theta_\a|=1\right\}$. For convenience
hereafter by $\theta\e$ we will understand a corresponding element
$(\theta\e_1,\dots,\theta\e_n)\in\mathbb{T}^n$.

We extract a subsequence (still denoted by $\eps$) such that
\begin{gather*}
\theta\e\underset{\eps\to
0}\to\theta=(\theta_1,\dots,\theta_n)\in\mathbb{T}^n
\end{gather*}
Let $u\e\in\mathrm{dom}(\Delta_{\M\e}^{\theta\e})$ be the
eigenfunction corresponding to $\lambda\e$, i.e.
$-\Delta_{\M\e}^{\theta\e}u\e=\lambda\e u\e$, $u\e\not=0$. We
normalize $u\e$ by the condition $\|u\e\|_{L_2(\M\e)}=1$, then
$\left\|\nabla u\e\right\|^2_{L_2(\M\e)}=\lambda\e$.

In order to describe the behaviour of $u\e$ as $\eps\to 0$ we need
some special operators. From now on by $C$ we denote a generic
constant independent of $\eps$.

We denote
$$\Omega_\square\e=\left\{\tilde x\in\Omega\e:\ x\in\square\setminus\left(\cupl_{i\in\I}\cupl_{j=1}^m D\e_{ij}\right)\right\}$$
and introduce an extension operator $\Pi\e: H^1(\M\e)\to
H^1(\square)$ such that for each $u \in H^1(\M\e)$:
\begin{gather}\label{pi_a}
\Pi\e u(x)=u(\tilde x)\ \text{ for }\ \tilde x\in
\Omega_\square\e\\\label{pi_b} \|\Pi\e u\|_{H^1(\square)}\leq
C\|u\|_{H^1(\Omega\e_\square)}
\end{gather}
It is known (see e.g. \citep[Chapter 4]{March}) that such an
operator exists.

By $\langle u\rangle_B$ we denote the average value of the
function $u$ over the domain $B\subset M\e$ ($|B|\not= 0$), i.e.
$\langle u \rangle_{B}=\ds{1\over |B|}\intl_{B}u dV\e$, where
$dV\e$ is the density of the Riemannian measure on $M\e$. The same
notation remains for $B\subset\mathbb{R}^n$.

If $\Sigma\subset M\e$ is a $(n-1)$-dimensional submanifold then
$g\e$ induces on $\Sigma$ the Riemannian \textcolor{black}{metric} and measure. We
denote by $dS\e$ the density of this measure. Again by $\langle
u\rangle_\Sigma$ we denote the average value of the function $u$
over $\Sigma$, i.e. $\langle u\rangle_B=\ds{1\over
|\Sigma|}\intl_{\Sigma}u dS\e$ (here $|\Sigma|=\intl_\Sigma
dS\e$).

We introduce the operators $\Pi_j\e: L_2(\M\e)\to L_2(\square)$
($j=1,\dots,m$) by the formula:
\begin{gather*}
i\in\I,\ x\in\square_i\e:\ \Pi_j\e u(x)=\langle u\e
\rangle_{B_{ij}\e}
\end{gather*}
\textcolor{black}{Recall} that
$\square=\cupl_{i\in\I}\square_i\e$\textcolor{black}{.} Using
\textcolor{black}{the Cauchy inequality} and (\ref{b_size}) we
obtain
\begin{gather}\label{pij}
\|\Pi_j\e u\|_{L_2(\square)}\leq
C\|u\|_{L_2\left(\cup_{i\in\I}\cup_{j=1}^m B_{ij}\e\right)}
\end{gather}

\textcolor{black}{In view of} (\ref{pi_b}), (\ref{pij}) the norms
$\|\Pi\e u\e\|_{H^1(\square)}$, $\|\Pi_j\e u\e\|_{L_2(\square)}$
($j=1,\dots,m$) are bounded uniformly in $\eps$. Using the
embedding theorem (see e.g. \citep[Chapter 4]{Taylor}) we obtain
that the sub-sequence (still denoted by $\eps$), the functions
$u\in H^1(\square)$, $u_j\in L_2(\square)$, $j=1,\dots,m$ exist
such that
\begin{gather*}
\Pi\e u\e\underset{\eps\to 0}\to u\text{ weakly in
}H^1(\square)\text{ and strongly in }L_2(\square) ,\quad\Pi_j\e
u\e\underset{\eps\to 0}\to u_j\text{ weakly in }L_2(\square)
\end{gather*}
Moreover due to the trace theorem (see e.g. \citep[Chapter
4]{Taylor}) $\Pi\e u\e$, $u\in L_2(\partial\square)$ and
\begin{gather}\label{trace}
\Pi\e u\e\underset{\eps\to 0}\to u\text{ strongly in
}L_2(\partial\square)
\end{gather}
Since
$u\e\in\mathrm{dom}(\Delta_{\M\e}^{\theta\e})$\textcolor{black}{,
then} in view of (\ref{faces})
\begin{gather*}
u\e(x+e_\a)=\overline{\theta\e_\a} u\e(x),\ \tilde x\in
\mathbf{M}_\a,\ \a=1,\dots,n
\end{gather*}
\textcolor{black}{Therefore,}
\begin{gather*}
u(x+e_\a)=\overline{\theta_\a} u(x),\ x\in \mathbf{M}_\a,\
\a=1,\dots,n
\end{gather*}
\textcolor{black}{Thus,} $u\in
\mathrm{dom}(\bar{\eta}^\theta_{\square})$. Recall (see
\textcolor{black}{Section \ref{sec1}}) that
$\bar{\eta}^\theta_{\square}$ is the sesquilinear form which
generates the operator $-\Delta_{\M\e}^{\theta\e}$.

We also need some auxiliary lemmas.
\begin{lemma}\label{lm-1}
For any $j=1,\dots,m$:
\begin{gather}\label{conv1}
\liml_{\eps\to 0}\eps^n\suml_{i\in\I}\left|\langle
u\e\rangle_{S_{ij}\e}\right|^2=\textcolor{black}{\|u\|^2_{L_2(\square)}}
\end{gather}
\end{lemma}

\begin{proof} We denote
$\hat R_{ij}\e=\left\{\tilde x\in\Omega\e:\
d_{j}\e+{\kappa\eps\over
4}\leq|x-x_{ij}\e|<d_{j}\e+{\kappa\eps\over 2}\right\}$. One has
the inequalities:
\begin{gather}\label{ineq1}0\leq\|\Pi\e u\e\|^2_{L_2(\square_i\e)}-\eps^n\left|\langle \Pi\e
u\e\rangle_{\square_{i}\e}\right|^2\leq C\eps^2\|\nabla \Pi\e
u\e\|^2_{L_2(\square_i\e)},\quad i\in\I\\\label{ineq2}
\left|\langle \Pi\e u\e\rangle_{\square_{i}\e}-\langle
u\e\rangle_{\hat R_{ij}\e}\right|^2\leq C\|\nabla \Pi\e
u\e\|^2_{L_2(\square_{i}\e)}\eps^{2-n},\quad i\in\I\\\label{ineq3}
\left|\langle u\e\rangle_{S_{ij}\e}-\langle u\e\rangle_{\hat
R_{ij}\e}\right|^2\leq C\|\nabla u\e\|^2_{L_2(\hat
R_{ij}\e)}\eps^{2-n},\quad i\in\I
\end{gather}
which are valid for any $u\e\in H^1(\Omega_\square\e)$, $
j=1,\dots,m$. \textcolor{black}{Inequality} (\ref{ineq1}) is the
Poincar\'{e} inequality, the inequality (\ref{ineq2}) follows directly
from \citep[Lemma 2.1]{Khrab5}, and the inequality (\ref{ineq3})
can be proved in the same way as \textcolor{black}{inequality} (2.2)
from \citep{Khrab6}.\footnote{In \citep{Khrab6}
\textcolor{black}{inequality} (\ref{ineq3}) with $\partial\hat
R_{ij}\e\setminus S_{ij}\e$ instead of $S_{ij}\e$ was proved. For
$S_{ij}\e$ the proof is
 similar. We remark that in the case $n=2$
 \textcolor{black}{inequality}
 (\ref{ineq3}) is valid with $|\ln\eps|$ instead of $\eps^{2-n}$.}

\textcolor{black}{Equality} (\ref{conv1}) follows directly from
(\ref{ineq1})-(\ref{ineq3}). \textcolor{black}{The lemma is proved.}

\end{proof}

\begin{lemma}\footnote{This result was given in
\citep{Khrab3} without a justification. In the current work we
present a complete proof.}\ \label{lm0}For $j=1,\dots,m$:
\begin{gather*}\liml_{\eps\to 0}\lambda_{1}^D(G_{ij}\e)=\sigma_j
\end{gather*}
where $\sigma_j$ is defined by \textcolor{black}{formula}
(\ref{sigma}).
\end{lemma}
\begin{proof} Let
$v_{ij}\e\in\mathrm{dom}(\Delta^D_{G_{ij}\e})$ be the
eigenfunction corresponding to $\lambda_1(G_{ij}\e)$ such that
$\langle v_{ij}\e\rangle_{B_{ij}\e}=1$. Instead of calculating
$v_{ij}\e$ in the exact form we construct a convenient
approximation $\mathbf{v}_{ij}\e$ for it.

We introduce the notations:
\begin{gather*}
\hat B_{ij}\e=\left\{\tilde x=(\theta_1,\dots,\theta_n)\in
B_{ij}\e:\ \theta_n\in [\Theta_j\e,{\pi/2}]\right\}\\
\hat G_{ij}\e=\hat B_{ij}\e\cup R_{ij}\e\\ \hat
S_{ij}\e=\left\{\tilde x=(\theta_1,\dots,\theta_n)\in B_{ij}\e:\
\theta_n={\pi/2}\right\}=\partial \hat B_{ij}\e\setminus\partial
B_{ij}\e
\end{gather*}

Let the function $\hat{v}_{ij}\e$ be the solution of
\textcolor{black}{the following boundary value problem}:
\begin{gather}\label{bvp}
-\Delta_{\hat G_{ij}\e} \hat{v}_{ij}\e=0\ \text{in}\ \hat
G_{ij}\e\\\label{bvp+} {\hat{v}}_{ij}\e|_{S_{ij}\e}=0,\quad
{\hat{v}}_{ij}\e|_{\hat S_{ij}\e}=1
\end{gather}
Here by $-\Delta_{\hat G_{ij}\e}$ we denote the operator which is
defined by the operation (\ref{local}) and the definitional domain
$\mathrm{dom}(\Delta_{\hat G_{ij}\e})=\left\{u: u=v|_{\hat
G_{ij}\e},\ v\in \mathrm{dom}(\Delta_{M\e})\right\}$. For
convenience \textcolor{black}{from now on} we use the notation
$-\Delta$ instead of $-\Delta_{\hat G_{ij}\e}$. It is easy to see
that the function $\hat v_{ij}\e$ is smooth in $R_{ij}\e$ and
$B_{ij}\e$, the limiting values of $\hat v_{ij}\e$ in the domains
$R_{ij}\e$ and $\hat B_{ij}\e$ coincide on $\partial B_{ij}\e$,
the normal derivatives \textcolor{black}{satisfy} the condition
${\partial \hat v_{ij}\e\over
\partial r}+{1\over b_j\e}{\partial \hat v_{ij}\e\over\partial
\theta_n}=0$.

Due to the symmetry of $\hat G_{ij}\e$ one can easily calculate
 $\hat v_{ij}\e$ (recall that we consider the case
$n\geq 3$):
\begin{gather}\label{v_exact}
\hat v_{ij}\e(\tilde x)=\begin{cases}\ds{\mathrm{A}_{j}\e |x-x_{ij}\e|^{2-n}}+\mathrm{B}_{j}\e,&\tilde x\in R_{ij}\e\\
\mathrm{C}_{j}\e F(\theta_n)+1,&\tilde
x=(\theta_1,\dots,\theta_n)\in \hat B_{ij}\e
\end{cases}
\end{gather}
where
$F(\theta_n)=\ds\intl_{\pi/2}^{\theta_n}({\sin^{1-n}\psi})d\psi$
and the constants $\mathrm{A}_{j}\e,\ \mathrm{B}_{j}\e,\
\mathrm{C}_{j}\e$ are defined by the formulae
\begin{gather}\label{v_exact_c}
\mathrm{A}_{j}\e=\ds{\left(d_{j}\e\right)^{n-2}\over
1-\left({2d_j\e\over\kappa\eps}\right)^{n-2}
-(n-2)F(\Theta_{j}\e)\left({d_j\e\over b_j\e}\right)^{n-2}},\quad
\mathrm{B}_{j}\e=-\ds{\mathrm{A}_{j}\e\over \left({\kappa\eps\over
2}\right)^{n-2}},\quad
\mathrm{C}_{j}\e=(n-2)\ds{\mathrm{A}_{j}\e\over (b_{j}\e)^{n-2}}
\end{gather}
We redefine $\hat v_{ij}\e$ by $1$ in $B_{ij}\e\setminus\hat
B_{ij}\e$ preserving the same notation.

Direct calculations lead to the following asymptotics as $\eps\to
0$:
\begin{gather}\label{v_hat_estim}
\|\nabla\hat
v_{ij}\e\|^2_{L_2(G_{ij}\e)}\sim\sigma_j\rho_j\eps^n,\quad\|\hat
v_{ij}\e\|^2_{L_2(G_{ij}\e)}\sim\rho_j\eps^n
\end{gather} where
$\sigma_j$, $\rho_j$ are defined by \textcolor{black}{formulae}
(\ref{sigma}), (\ref{rho}).

We define the function
$\mathbf{v}_{ij}\in\mathrm{dom}(\Delta^D_{G_{ij}\e})$ by the
formula
\begin{gather}\label{v_bf}
\mathbf{v}_{ij}\e(\tilde x)=\begin{cases}\hat v_{ij}\e&\tilde x\in R_{ij}\e,\\
1+(\hat v_{ij}\e(\tilde
x)-1)\Phi\ds\left({\theta_n}\right),&\tilde
x=(\theta_1,\dots,\theta_n)\in \hat B_{ij}\e\\1,&\tilde x\in
B_{ij}\e\setminus \hat B_{ij}\e
\end{cases}
\end{gather}
Here $\Phi(\theta_n)$  is a twice continuously differentiable
non-negative function on $[0,\infty)$ equal to $1$ as $0\leq
\theta_n\leq{\pi/4}$ and equal to $0$ as $\theta_n\geq{\pi/2}$. We
have the following asymptotics as $\eps\to 0$:
\begin{gather}\label{v_bf_estim}
\|\nabla\mathbf{v}_{ij}\e\|_{L_2(G_{ij}\e)}^2\sim\|\nabla\hat
v_{ij}\e\|^2_{L_2(G_{ij}\e)},\quad
\|\mathbf{v}_{ij}\e\|_{L_2(G_{ij}\e)}^2\sim\|\hat
v_{ij}\e\|^2_{L_2(G_{ij}\e)},\quad \|\Delta
\mathbf{v}_{ij}\e\|_{L_2(G_{ij}\e)}^2={O}(\eps^n)\\\label{v_bf_estim+}
\liml_{\eps\to 0}\eps^{-n}\left(\|
\mathbf{v}_{ij}\e-1\|^2_{L_2(B_{ij}\e)}+\|
\mathbf{v}_{ij}\e\|^2_{L_2(R_{ij}\e)}\right)=0
\end{gather}
It follows from \textcolor{black}{the min-max principle} (see e.g.
\citep{Reed}) that
\begin{gather}\label{courant}
\lambda_{1}(G_{ij}\e)=\ds{\|\nabla
v_{ij}\e\|_{L_2(G_{ij}\e)}^2\over\|v_{ij}\e\|_{L_2(G_{ij}\e)}^2}\leq
\ds{\|\nabla
\mathbf{v}_{ij}\e\|_{L_2(G_{ij}\e)}^2\over\|\mathbf{v}_{ij}\e\|_{L_2(G_{ij}\e)}^2}
\end{gather}
\textcolor{black}{Note, that this automatically gives the inequality
$\liml_{\eps\to 0}\lambda_{1}(G_{ij}\e)\leq\sigma_j$.}

We present the eigenfunction $v_{ij}\e$ in the form
\begin{gather}
\label{v_repres}v_{ij}\e=\mathbf{v}_{ij}\e+w_{ij}\e
\end{gather}
Let us estimate the remainder $w_{ij}\e$. One has the following
estimates for the eigenfunction $v_{ij}\e$ (for the proof see
\citep[Lemma 4.2]{Khrus2}):
\begin{gather}
\label{estim_v3}
\|v_{ij}\e\|_{L_2(G_{ij}\e)}^2=\|v_{ij}\e\|_{L_2(B_{ij}\e)}^2+O(\eps^{n+2})=|B_{ij}\e|+
O(\eps^{n+2})
\\
\label{estim_v2} \|v_{ij}\e\|^2_{L_2(R_{ij}\e)}\leq C\eps^{n+2}
\end{gather}
Using (\ref{v_bf_estim+}), (\ref{estim_v3}), (\ref{estim_v2}) we
obtain
\begin{gather}\label{w_estim}
\eps^{-n}\|w_{ij}\e\|^2_{L_2(G_{ij}\e)}\leq
2\eps^{-n}\left(\|\mathbf{v}_{ij}\e\|_{L_2(R_{ij}\e)}^2+\|v_{ij}\e\|_{L_2(R_{ij}\e)}^2+
\|1-\mathbf{v}_{ij}\e\|_{L_2(B_{ij}\e)}^2+\|{v}_{ij}\e-1\|_{L_2(B_{ij}\e)}^2\right)\underset{\eps\to
0}\to 0
\end{gather}

Substituting (\ref{v_repres}) into (\ref{courant}) and integrating
by parts we get
\begin{gather}\label{subst1}
\|\nabla w\e_{ij}\|_{L_2(G_{ij}\e)}^2\leq 2\left|(\Delta
\mathbf{v}_{ij}\e,w\e)_{L_2(G_{ij}\e)}\right|+
\|\nabla\mathbf{v}_{ij}\e\|^2_{L_2(G_{ij}\e)}\left(\ds{\|v_{ij}\|^2_{L_2(G_{ij}\e)}
\over\|\mathbf{v}_{ij}\e\|_{L_2(G_{ij}\e)}}-1\right)
\end{gather}
Taking into account (\ref{v_hat_estim}), (\ref{v_bf_estim}),
(\ref{estim_v3}), (\ref{w_estim}) we conclude that (\ref{subst1})
implies
\begin{gather}\label{nablaw_estim}
\eps^{-n}\|\nabla w\e_{ij}\|_{L_2(G_{ij}\e)}^2\underset{\eps\to
0}\to 0
\end{gather}
It follows from (\ref{v_hat_estim}), (\ref{v_bf_estim}),
(\ref{w_estim}), (\ref{nablaw_estim}) that $\liml_{\eps\to
0}\lambda_1\e(G_{ij}\e)=\sigma_j$. \textcolor{black}{The lemma is proved.}
\end{proof}

\begin{lemma}\label{lm0+}For $j=1,\dots,m$:
\begin{gather*}
\liml_{\eps\to 0}\lambda_{2}^D(G_{ij}\e)=\infty
\end{gather*}
\end{lemma}

\begin{proof}
Let $\mathbf{G}_{j}\e$ be an $n$-dimensional surface embedded into
$\mathbb{R}^{n+1}$ (below $x\in\mathbb{R}^n,\ z\in\mathbb{R})$:
\begin{gather*}
\mathbf{G}_{j}\e=\mathbf{R}_{j}\e\cup \mathbf{B}_{j}\e
\end{gather*}
where
\begin{gather*}
\mathbf{R}_{j}\e=\left\{(x,z)\in\mathbb{R}^{n+1}:\ \eps^{-1}d_{j}\e\leq|x|< \kappa/2,\ z=0\right\}\\
\mathbf{B}_{j}\e=\left\{(x,z)\in\mathbb{R}^{n+1}:\
|x|^2+\left(z-b_j\cos\Theta_j\e\right)^2=(b_{j})^2,\ z\geq
0\right\}
\end{gather*}
We \textcolor{black}{equip $\mathbf{G}_{j}\e$ with} the Riemannian
\textcolor{black}{metric} induced by \textcolor{black}{the Euclidean}
\textcolor{black}{metric} in $\mathbb{R}^{n+1}$. By $dV$ we denote
the density of the Riemannian measure on $\mathbf{G}_{j}\e$.
\textcolor{black}{Thus,} $\mathbf{G}_{j}\e$ is the
$\eps^{-1}$-homothetic image of $G_{ij}\e$.

Evidently one has the following relation between the spectra of
$-\Delta^D_{G_{ij}\e}$ and $-\Delta^D_{\mathbf{G}_j\e}$:
\begin{gather}\label{relation}
\forall k\in\mathbb{N}:\
\lambda_{k}^D(G_{ij}\e)=\eps^{-2}\lambda_k^D(\mathbf{G}_{ij}\e)
\end{gather}

We denote
\begin{gather*}
\mathbf{R}=\left\{(x,z)\in\mathbb{R}^{n+1}:\ |x|< \kappa/2,\
z=0\right\},\quad \mathbf{B}_{j}=\left\{(x,z)\in\mathbb{R}^{n+1}:\
|x|^2+z^2=(b_{j})^2\right\}
\end{gather*}
Further we will prove that
\begin{gather}\label{conv2}
\forall k\in\mathbb{N}:\quad
\lambda_k^D(\mathbf{G}_{j}\e)\underset{\eps\to 0}\to\lambda_k
\end{gather}
where $\left\{\lambda_k\right\}_{k\in\mathbb{N}}$ are the
eigenvalues of the operator $\mathcal{L}_j$ which acts in the
space $L_2(\mathbf{R})\oplus~L_2(\mathbf{B}_j)$ and is defined by
the formula
\begin{gather*}
\mathcal{L}_j=-\left(\begin{matrix}\Delta^D_{\mathbf{R}}&0\\0&\Delta_{\mathbf{B}_j}\end{matrix}\right)
\end{gather*}
Here the eigenvalues are renumbered in the increasing order and
with account of their multiplicity.

One has $\lambda_1=\lambda_1(\mathbf{B}_j)=0$,
$\lambda_2=\minl\left\{\lambda^D_1(\mathbf{R}),\lambda_2(\mathbf{B}_j)\right\}>0$.
\textcolor{black}{Therefore,} in view of
(\ref{relation})-(\ref{conv2}) $\liml_{\eps\to
0}\lambda_2^D(G_{ij}\e)=\infty$. \textcolor{black}{Thus,} to
complete \textcolor{black}{the proof of the lemma} we have to prove
(\ref{conv2}). For that we use \textcolor{black}{the abstract
scheme} proposed in the work \citep{IOS}.\medskip

\noindent\textbf{Theorem} \citep{IOS}. {\it Let $\mathcal{H}^\eps,
\mathcal{H}^0$ be separable Hilbert spaces, let
$\mathcal{A}\e:\mathcal{H}\e\to \mathcal{H}\e,\
\mathcal{A}^0:\mathcal{H}^0 \to \mathcal{H}^0$ be linear
continuous operators,
$\mathrm{im}\mathcal{A}^0\subset\mathcal{V}\subset \mathcal{H}^0$,
where $\mathcal{V}$ is a subspace in $\mathcal{H}^0$.

Suppose that the following conditions $C_1-C_4$ hold:

{$C_1.$} The linear bounded operators $R\e:\mathcal{H}^0\to
\mathcal{H}\e$ exist such that $ \|R\e
f\|^2_{\mathcal{H}\e}\underset{\eps\to 0}\to
\gamma\|f\|^2_{\mathcal{H}^0}$ for any $f\in \mathcal{V}$. Here
$\gamma>0$ is a constant.

{$C_2.$} Operators $\mathcal{A}\e, \mathcal{A}^0$ are positive,
compact and self-adjoint. The norms
$\|\mathcal{A}\e\|_{\mathcal{L}(\mathcal{H}\e)}$ are bounded
uniformly in $\eps$.

{$C_3.$} For any $f\in \mathcal{V}$: $\|\mathcal{A}\e R\e
f-R\e\mathcal{A}^0 f\|_{\mathcal{H}\e}\underset{\eps\to 0}\to 0$.

{$C_4.$} For any sequence $f\e\in \mathcal{H\e}$ such that
$\sup\limits_{\eps} \|f\e\|_{\mathcal{H}\e}<\infty$ the
subsequence $\eps^\prime\subset\eps$ and $w\in \mathcal{V}$ exist
such that $ \|\mathcal{A}\e f\e-R\e
w\|_{\mathcal{H}\e}\underset{\eps=\eps^\prime\to 0}\longrightarrow
0$.

Then for any $k\in\mathbb{N}$\ $$\mu_k\e\underset{\eps\to
0}\to\mu_k$$ where $\{\mu_k\e\}_{k=1}^\infty$ and
$\left\{\mu_k\right\}_{k=1}^\infty$ are the eigenvalues of the
operators $\mathcal{A}\e$ and $\mathcal{A}^0$, which are
renumbered in the increasing order and with account of their
multiplicity. }
\medskip

Let us apply this theorem. We set
$\mathcal{H}\e=L_2(\mathbf{G}_j\e)$,
$\mathcal{H}^0=L_2(\mathbf{R})\oplus L_2(\mathbf{B}_j)$,
$\mathcal{A}\e=(-\Delta_{\mathbf{G}_j\e}^D+\mathrm{I})^{-1}$,
$\mathcal{A}^0=(\mathcal{L}_j+\mathrm{I})^{-1}$,
$\mathcal{V}=\mathcal{H}^0$. We introduce the operator
$R\e:\mathcal{H}^0\to \mathcal{H}\e$ by the formula:
\begin{gather*}
[R\e f](x,z)=\begin{cases}f^R(x),&(x,0)\in \mathbf{R}_j\e,\\
f^B(x,z-b_j\e\cos\Theta_j\e),&(x,z)\in \mathbf{B}_j\e,
\end{cases}\quad
f=(f^R,f^B)\in \mathcal{H}^0=L_2(\mathbf{R})\oplus
L_2(\mathbf{B}_j)
\end{gather*}

We also denote $H_0^1(\mathbf{R})=\left\{u\in H^1(\mathbf{R}):\
u|_{\partial \mathbf{R}}=0\right\}$,
$H_0^1(\mathbf{G}_j\e)=\left\{u\in H^1(\mathbf{G}_j\e):\
u|_{\partial \mathbf{G}_j\e}=0\right\}$,
$\mathcal{H}^1=H_0^1(\mathbf{R})\oplus H^1(\mathbf{B}_j)\subset
\mathcal{H}^0$ and introduce the operator
$Q\e:H_0^1(\mathbf{G}\e_j)\to \mathcal{H}^1$ satisfying the
properties that are similar to those of the operator $\Pi\e$ (see
above):
\begin{gather}\label{Q}
\forall\eps>0,\ \forall v\in {H}_0^1(\mathbf{G}_j\e):\ R\e Q\e
v=v,\quad \|Q\e v\|_{\mathcal{H}^1}\leq
C\|v\|_{H_0^1(\mathbf{G}\e_j)}
\end{gather}

Evidently \textcolor{black}{conditions $C_1$} (with $\gamma=1$) and
$C_2$ hold. We verify \textcolor{black}{condition} $C_3$. Let
$f\in\mathcal{H}\e$. We set $f\e=R\e f$, $v\e=\mathcal{A}\e f\e$,
$\hat v\e=Q\e v\e$. One has
\begin{gather}
\label{int_ineq_G} \intl_{\mathbf{G}_j\e}\left((\nabla v\e,\nabla
w\e)+u\e w\e-f\e w\e\right)dV=0,\quad \forall w\e\in
H_0^1(\mathbf{G}_j\e)
\end{gather}

Clearly the norms $\|v\e\|^2_{H_0^1(\mathbf{G}\e_j)}$ are bounded
uniformly in $\eps$. Taking into account (\ref{Q}) we conclude
that the subsequence (still denoted by $\eps$) and $v=(v^R,v^B)\in
\mathcal{H}^1$ exist such that
$$\hat v\e=({\hat v}^{\eps R},{\hat v}^{\eps B})
\underset{\eps\to 0}\to v\text{ weakly in }\mathcal{H}^1\text{ and
strongly in }\mathcal{H}^0$$

Let $w\in\widehat{\mathcal {H}}^1=\left\{w=(w^R,w^B)\in
\mathcal{H}^1: \supp f^R\subset
\mathbf{R}\setminus\big\{(0,0)\big\},\ \supp f^B\subset
{\mathbf{B}_j}\setminus\big\{(0,-b_j)\big\}\right\}$, i.e. $w^R=0$
in a neighbourhood of $\big\{(0,0)\big\}$, $w^B=0$ in a
neighbourhood of $\big\{(0,-b_j)\big\}$. We set $w\e=R\e w$.
Then\textcolor{black}{, when $\eps$ is small enough,} $w\e=0$ in
some neighbourhood of $\partial \mathbf{B}_{j}\e$ and $w\e\in
H_0^1(\mathbf{G}_j\e)$. Substituting $w\e$ into (\ref{int_ineq_G})
we obtain ($\eps$ is small enough):
\begin{gather}
\label{int_ineq_G1}\intl_{\mathbf{R}}\left((\nabla {\hat v}^{\eps
R},\nabla w^R)+{\hat v}^{\eps R} w^R-f^R w^R\right)dx+
\intl_{\mathbf{B}_j}\left((\nabla {\hat v}^{\eps B},\nabla
w^B)+{\hat v}^{\eps B} w^B-f^Bw^B\right)dV=0
\end{gather}
Passing to the limit in (\ref{int_ineq_G1}) as $\eps\to 0$ and
taking into account that the space $\widehat{\mathcal {H}}^1$ is
dense in $\mathcal{H}^1$ (see e.g. \citep{Rauch})\textcolor{black}{,
we obtain} the equality $\mathcal{A}^0 f=v$ that obviously implies
the fulfilment of $C_3$.

\textcolor{black}{Finally, condition} $C_4$ follows from the fact
that if $\sup\limits_{\eps} \|f\e\|_{\mathcal{H}\e}<\infty$ then
the norms $\|Q\e \mathcal{A}\e f\e\|_{\mathcal{H}^1}$ are bounded
uniformly in $\eps$ and\textcolor{black}{, therefore,} the
subsequence $\eps^\prime\subset\eps$ and $w\in \mathcal{H}^1$
exist such that
$$Q\e \mathcal{A}\e f\e\underset{\eps=\eps'\to 0}\longrightarrow w\text{ weakly in }\mathcal{H}^1\text{ and strongly in }\mathcal{H}^0$$

\textcolor{black}{Thus,} the eigenvalues $\mu_k\e$ of the operator
$\mathcal{A}\e$ \textcolor{black}{converge} to the eigenvalues
$\mu_k$ of the operator $\mathcal{A}^0$ as $\eps\to 0$. But
$\lambda_k^D(\mathbf{G}_{ij}\e)=(\mu_k\e)^{-1}-1$,
$\lambda_k=(\mu_k)^{-1}-1$ that implies (\ref{conv2}).
\textcolor{black}{The lemma is proved}.

\end{proof}

\begin{lemma}\label{lm3} For $j=1,\dots,m$:
\begin{gather}\label{u}
\liml_{\eps\to 0}\suml_{i\in\I}
\|u\e\|^2_{L_2(B_{ij}\e)}=\rho_j\left(\sigma_j\over\sigma_j-\lambda\right)^2
\|u\|_{L_2(\square)}^2
\end{gather}
\end{lemma}

\begin{proof} For $\tilde x\in G_{ij}\e$ we denote $l\e(\tilde x)=\dist_{g\e}(\tilde
x,S_{ij}\e)$, where by $\dist_{g\e}(\cdot,\cdot)$ we denote the
distance with respect to the \textcolor{black}{metric} $g\e$. We introduce the set
$$S_{ij}\e[\tilde x]=\left\{\tilde y\in G_{ij}\e:\ l\e(\tilde y)=l\e(\tilde x)\right\}$$
Obviously $S_{ij}\e[\tilde x]$ is a $(n-1)$-dimensional sphere (in
particular if $\tilde x\in
\partial B_{ij}\e$ then $l\e(\tilde x)=\kappa\eps/2$ and $S_{ij}\e[\tilde x]=\partial
B_{ij}\e$).

We define the function $u_{ij}\e(\tilde x)$ by the formula:
\begin{gather*}
u_{ij}\e(\tilde x)=\langle u\e\rangle_{S_{ij}\e[\tilde x]},\
\tilde x\in G_{ij}\e
\end{gather*}
\textcolor{black}{Using }the Poincar\'{e} inequality (for the spheres
$S_{ij}[\tilde x]$) we get
\begin{gather}\label{poincare}
\suml_{i\in\I}\left\|u\e-u_{ij}\e\right\|^2_{L_2(G_{ij}\e)}\leq
C\suml_{i\in\I}\maxl_{\tilde x\in
G_{ij}\e}\big(\mathrm{diam}S_{ij}\e[\tilde x]\big)^2\|\nabla
u\e\|^2_{L_2(G_{ij}\e)}\leq C\eps^2 \|\nabla u\e\|^2_{L_2(\M\e)}
\end{gather}

We denote ${\mathbf{u}}_{ij}\e=u_{ij}\e-\langle
u\e\rangle_{S_{ij}\e}$. Clearly
$\mathbf{u}_{ij}\e\in\mathrm{dom}(\Delta_{G_{ij}\e}^D)$ and
\begin{gather*}
-\Delta_{G_{ij}\e}^D \mathbf{u}_{ij}\e-\lambda\e
\mathbf{u}_{ij}\e=\lambda\e\langle u\e\rangle_{S_{ij}\e}
\end{gather*}
In view of Lemmas \ref{lm0}, \ref{lm0+} and since
$\lambda\notin\cupl_{j=1}^m\{\sigma_j\}$,
$\lambda\e\notin\sigma(-\Delta_{G_{ij}\e}^D)$ \textcolor{black}{when
$\eps$ is small enough}. \textcolor{black}{Therefore,} the following
expansion is valid:
\begin{gather}\label{decomp1}
\mathbf{u}_{ij}\e=\suml_{k=1}^\infty I^k_{ij}(\eps),\text{ where }
I^k_{ij}(\eps)={v^D_k(G_{ij}\e)\quad\over
\left\|v^D_k(G_{ij}\e)\right\|_{L_2(G_{ij}\e)}^2}\cdot{\left(f_{ij}\e,
v^D_k(G_{ij}\e)\right)_{L_2(G_{ij}\e)}\over
\left(\lambda^D_k(G_{ij}\e)-\lambda\e\right)}
\end{gather}
Here $f_{ij}\e=\lambda\e\langle u\e\rangle_{S_{ij}\e}$,
$\left\{v^D_k(G_{ij}\e)\right\}_{k=1}^m$ is a system of
\textcolor{black}{the eigenfunctions} of $-\Delta^D_{G_{ij}\e}$
corresponding to $\left\{\lambda^D_k(G_{ij}\e)\right\}_{k=1}^m$
and such that
$\left(v^D_k(G_{ij}\e),v^D_l(G_{ij}\e)\right)_{L_2(G_{ij}\e)}=0$
if $k\not= l$.

We  denote
$\Lambda\e=\maxl_{j=\overline{1,m}}\maxl_{k=\overline{2,\infty}}\left|\lambda\e-\lambda_k^D(G_{ij}\e)\right|^{-2}$.
Thus, it follows from Lemma \ref{lm0+} that $\liml_{\eps\to
0}\Lambda\e=0$. \textcolor{black}{Therefore,} taking into account
(\ref{b_size}) and using Lemma \ref{lm-1} we obtain
\begin{gather}\label{sum2345}
\suml_{i\in\I}\left\|\suml_{k=2}^\infty
I_{ij}^k(\eps)\right\|^2_{L_2(B_{ij}\e)}\leq\Lambda\e
\suml_{i\in\I}\|f_{ij}\e\|^2_{L_2(G_{ij}\e)}\leq
C(\lambda\e)^2\Lambda\e\suml_{i\in\I}\left|\langle
u\e\rangle_{S_{ij}\e}\right|^2\eps^n\underset{\eps\to 0}\to 0
\end{gather}

As in Lemma \ref{lm0} we denote $v_{ij}\e=v^D_1(G_{ij}\e)$. We
normalize $v_{ij}\e$ by the condition $\langle
v_{ij}\e\rangle_{B_{ij}\e}=1$. Using the estimates
(\ref{estim_v3}), (\ref{estim_v2}) and Lemma \ref{lm0} we obtain
that
\begin{gather}\label{sum1}
\suml_{i\in\I}\left\|I_{ij}^1(\eps)\right\|^2_{L_2(B_{ij}\e)}\sim
\suml_{i\in \I}{\lambda^2\rho_j\eps^n\left|\langle
u\e\rangle_{S_{ij}\e}\right|^2\over(\sigma_j-\lambda)^2}\sim{\lambda^2\rho_j\|u\|^2_{L_2(\square)}\over(\sigma_j-\lambda)^2}
\end{gather}
as $\eps\to 0$. \textcolor{black}{Thus,} it follows from
(\ref{decomp1})-(\ref{sum1}) that
\begin{gather}\label{v_lim}
\liml_{\eps\to
0}\suml_{i\in\I}\left\|\mathbf{u}_{ij}\e\right\|^2_{L_2(B_{ij}\e)}=
{\lambda^2\rho_j\|u\|^2_{L_2(\square)}\over(\sigma_j-\lambda)^2}
\end{gather}

\textcolor{black}{Finally,} using (\ref{decomp1}), (\ref{sum2345}),
(\ref{v_lim}) and Lemma \ref{lm-1} we get
\begin{gather}\notag
\suml_{i\in\I}\left\|u_{ij}\e\right\|^2_{L_2(B_{ij}\e)}=\suml_{i\in\I}\left(\left\|
\mathbf{u}_{ij}\e\right\|^2_{L_2(B_{ij}\e)}+2\langle
u\e\rangle_{S_{ij}\e}\intl_{B_{ij}\e}\mathbf{u}_{ij}\e(\tilde x)
dV\e+\left|\langle u\e\rangle_{S_{ij}\e}\right|^2\cdot
|B_{ij}\e|\right)\underset{\eps\to
0}\to\\\label{last}\underset{\eps\to
0}\to\left[{\lambda^2\rho_j\over(\sigma_j-\lambda)^2}+{2\lambda\rho_j\over\sigma_j-\lambda}+\rho_j\right]\|u\|^2_{L_2(\square)}=
\rho_j\left(\sigma_j\over\sigma_j-\lambda\right)^2\|u\|_{L_2(\square)}^2
\end{gather}
Then (\ref{u}) follows from (\ref{poincare}) and (\ref{last}).
\textcolor{black}{The lemma is proved}.
\end{proof}

\begin{lemma}\label{lm2} For any $w\in C^\infty_{\theta}(\square)$ the function $\hat w\e\in C^\infty(\square)$ exists such that:
\begin{gather}\label{w_1}
w+\hat w\e\in C^\infty_{\theta\e}(\square)\\\label{w_2}
\max_{x\in\square}\left|\hat
w\e(x)\right|+\max_{x\in\square}\left|\nabla \hat
w\e(x)\right|\underset{\eps\to 0}\to 0
\end{gather}
\end{lemma}

\begin{proof} We define the function $\mathbf{1}\e\in
C^\infty(\mathbb{R}^n)$ by the following recurrent formulae:
\begin{gather*}
\mathbf{1}\e(x_1,\dots,x_{n})=A_n(x_1,\dots,x_{n-1}) x_{n}+B_n(x_1,\dots,x_{n-1}),\\
\a=2,\dots,n:\
\begin{cases}B_\a(x_1,\dots,x_{\a-1})=A_{\a-1}(x_1,\dots,x_{\a-2})
x_{\a-1}+B_{\a-1}(x_1,\dots,x_{\a-2}),\\
A_\a(x_1,\dots,x_{\a-1})=\big(\overline{\theta_\a\e}/{\overline{\theta_\a}}-1\big)B_\a(x_1,\dots,x_{\a-1}),\end{cases}\\
B_1=1,\ A_1=\overline{\theta_1\e}/{\overline{\theta_1}}-1.
\end{gather*}
It is easy to see that
$\maxl_{x\in\square}\left|\mathbf{1}\e(x)-1\right|+\maxl_{x\in\square}\left|\nabla
\mathbf{1}\e(x)\right|\underset{\eps\to 0}\to 0$ and
$\mathbf{1}^\eps\in C^\infty_{\theta\e/\theta}(\square)$, where
$\theta\e/\theta:=(\theta_1\e/\theta_1,\dots,\theta_n\e/\theta_n)$.
Then we set
$$\hat w\e=(\mathbf{1}\e-1)w$$
Obviously the function $\hat w\e$ satisfies the conditions
(\ref{w_1}), (\ref{w_2}). \textcolor{black}{The lemma is proved.}
\end{proof}

We continue the \textcolor{black}{proof of Theorem 2.1}. For an
arbitrary $w\e\in \mathrm{dom}(\bar\eta_{\M\e})$ we have
\begin{gather}\label{int_ineq}
\intl_{\M\e}\left((\nabla u\e,\nabla w\e)_\eps-\lambda\e u\e
w\e\right)dV\e=0
\end{gather}
where $(\nabla u\e,\nabla w\e)_\eps$ is the scalar product of the
vectors $\nabla u\e$ and $\nabla w\e$ with respect to the \textcolor{black}{metric}
$g\e$.

We substitute into (\ref{int_ineq}) the test function $w\e$ of a
special type. Namely, let $w$ be \textcolor{black}{an arbitrary}
function from $C^\infty_\theta(\square)$, $\hat w\e\in
C^\infty(\square)$ be the function satisfying (\ref{w_1}),
(\ref{w_2}). Let $w_j$, $j=1,\dots,m$ be arbitrary functions from
$C^\infty(\square)$. Let $\Phi(r)$ be a twice continuously
differentiable non-negative function equal to $1$ as $0\leq
r\leq{1/4}$ and equal to $0$ as $r\geq{1/2}$. We set
\begin{gather*}
\widehat\Phi_{ij}\e=\Phi\left({|x-x_{ij}\e|-d_j\e\over
d_{j}\e}\right),\quad
\Phi_{ij}\e=\Phi\left({|x-x_{ij}\e|-d_j\e\over \kappa\eps}\right)
\end{gather*}
Then we set $w\e=\mathbf{w}\e+\delta\e$, where
\begin{gather}\label{w_bf}
\mathbf{w}\e(\tilde x)=\begin{cases}w(x),&\tilde x\in
\Omega_\square\e\setminus\left(\cupl_{i\in \I}\cupl_{j=1}^m
R_{ij}\e\right)\\w(x)+\ds\left(w(x\e_{ij})-w(x)
\right)\widehat\Phi_{ij}\e(x)+\\\qquad+\ds
\big(w_j(x_{ij}\e)-w(x_{ij}\e)\big)\mathbf{v}_{ij}\e(x)\Phi_{ij}\e(x),&\tilde x\in  R_{ij}\e\\
w_j(x_{ij}\e)+\left(w(x_{ij}\e)-w_j(x_{ij}\e)\right)\left(1-\mathbf{v}_{ij}\e
(\tilde x)\right),&\tilde x\in B_{ij}
\end{cases}\\\notag
\delta\e(\tilde x)=\begin{cases}\hat w(x),&\tilde x\in
\Omega_\square\e\setminus\left(\cupl_{i\in\I}\cupl_{j=1}^m
R_{ij}\e\right)\\\hat w\e(x)+\ds\left(\hat w\e(x\e_{ij})-\hat
w\e(x) \right)\widehat\Phi_{ij}\e(x),&\tilde x\in  R_{ij}\e\\
\hat w\e(x_{ij}\e),&\tilde x\in B_{ij}\e
\end{cases}
\end{gather}
Here the function $\mathbf{v}_{ij}\e$ is defined by (\ref{v_bf}),
(\ref{v_exact}), (\ref{v_exact_c}). It follows from (\ref{w_1})
that $w\e\in \mathrm{dom}(\bar\eta_{\M\e})$.

Substituting this $w\e$ into (\ref{int_ineq}) and integrating by
parts we obtain
\begin{gather}\label{int_ineq_mod}
\ds\intl_{\M\e}\left(-u\e \Delta\mathbf{w}\e-\lambda\e u\e
w\e\right)dV\e+\ds\intl_{\partial
\M\e}\nu\left[\mathbf{w}\e\right]u\e dS\e+
\intl_{\M\e}\left((\nabla u\e,\nabla\delta\e)_\eps-\lambda\e u\e
\delta\e\right)dV\e=0
\end{gather}
where ${\nu}$ is the outward normal vector field on $\partial
\M\e$.

In view of (\ref{d_size})-(\ref{b_size}) and the Cauchy
inequality\textcolor{black}{, the last term} in (\ref{int_ineq_mod}) is estimated by\\
$C\|u\e\|_{H^1(\M\e)}\sqrt{\maxl_{x\in\square}\left|\hat{w}(x)\right|^2+
\maxl_{x\in\square}\left|\nabla \hat{w}(x)\right|^2}$ and tends to
zero as $\eps\to 0$ in view of (\ref{w_2}).

In view of (\ref{trace}) the second term tends to
$\ds\int_{\partial\square}\nu\left[w\right]u ds$ as $\eps\to 0$,
where ${\nu}$ is the outward normal vector field on $\partial
\square$, $ds$ is the density of the Lebesgue measure on
$\partial\square$.

Now let us investigate the first term. Firstly we study the
integrals over $\Omega_\square\e$. Integrating by parts we get
\begin{gather}
\notag
\left|\suml_{i\in\I}\suml_{j=1}^m\intl_{R_{ij}\e}-\Delta\left\{
\big(w(x_{ij}\e)-w(x)\big)\widehat\Phi_{ij}\e(x)\right\}u\e(x)dV\e\right|=\\\notag=
\left|\suml_{i\in\I}\suml_{j=1}^m\left(\intl_{R_{ij}\e\cup
D_{ij}\e}\bigg(\nabla\left\{
\big(w(x_{ij}\e)-w(x)\big)\widehat\Phi_{ij}\e(x)\right\},\nabla\e
\Pi\e u\e(x)\bigg)
dx-\intl_{D_{ij}\e}\Delta w \Pi\e u\e dx\right)\right|\leq\\
\label{estim1}\leq C(w)\cdot\|\Pi\e
u\e\|_{H^1(\square)}\cdot\sqrt{\suml_{i\in\I}\suml_{j=1}^m\left|D_{ij}\e\cup\mathrm{supp}
\left[\nabla\widehat\Phi\e_{ij}\right]\right|}\underset{\eps\to
0}\to 0
\end{gather}
\textcolor{black}{Hereafter by }$C(w)$ we denote a constant
depending only on $w$\textcolor{black}{.}

Let us prove that the function $\xi\e\in L_2(\square)$,
$$
\xi\e(x)=\begin{cases}\ds\suml_{i\in\I}\suml_{j=1}^m-\Delta\left\{
\big(w_{j}\e(x_{ij}\e)-w(x_{ij}\e)\big)\mathbf{v}\e_{ij}(x)\Phi_{ij}\e(x)\right\},&x\in
R\e_{ij}\\0,&x\in\square\setminus \cupl_{i\in\I}\cupl_{j=1}^m
R\e_{ij}
\end{cases}
$$ converges
weakly in $L_2(\square)$ to the function
$\ds\suml_{j=1}^m\sigma_j\rho_j(w-w_j)$. Indeed using the
properties of $\mathbf{v}_{ij}\e$
\begin{gather*}
x\in R_{ij}\e:\ \Delta \mathbf{v}_{ij}\e(x)=0,\quad |D^\a
\mathbf{v}_{ij}\e(x)|\leq \ds {C\eps^n |x-x_{ij}\e|^{2-n-|a|}},\
\alpha=0,1
\end{gather*}
and the enclosure
$\mathrm{supp}(D^\a\Phi_{j}\e)\subset\left\{x\in\Omega_\square\e:\
{\kappa\eps/4}\leq |x-x_{ij}\e|\leq{\kappa\eps/2}\right\}$
($\a\not=0$) we obtain
\begin{gather}
\label{lim4} \intl_{R_{ij}\e}\left|-\Delta\left\{
\big(w_j(x_{ij}\e)-w(x_{ij}\e)\big)\mathbf{v}_{ij}\e(
x)\Phi_{ij}\e(x)\right\}\right|^2dx<C(w)\eps^n
\end{gather}
Hence the norms $\|\xi\e\|_{L_2(\square)}$ are bounded uniformly
in $\eps$. Taking into account (\ref{lim4}) we obtain for an
arbitrary $f\in C^\infty(\square)$ (below $\nu\e$ is the normal
vector field on $\partial D_{ij}\e$ directed outward $R_{ij}\e$):
\begin{multline}\label{lim3}
 \suml_{i\in\I}\suml_{j=1}^m\intl_{R_{ij}\e}-\Delta\left\{
\big(w_j(x_{ij}\e)-w(x_{ij}\e)\big)\mathbf{v}_{ij}\e(
x)\Phi_{ij}\e(x)\right\}f(x)dV\e=\\= \suml_{i\in\I}\suml_{j=1}^m
f(x_{ij}\e)\big(w(x_{ij}\e)-w_j(x_{ij}\e)\big)\intl_{\partial
D_{ij}\e}\nu\e\left[{\mathbf{v}_{ij}\e}\right]dS\e+\bar o(1)=\\
= \suml_{i\in\I}\suml_{j=1}^m
f(x_{ij}\e)(w(x_{ij}\e)-w_j(x_{ij}\e))\sigma_j\rho_j\eps^n+\bar
o(1)\underset{\eps\to 0}\to
\suml_{j=1}^m\sigma_j\rho_j\intl_{\square}f(x)(w(x)-w_j(x))dx
\end{multline}
Here we have used the following computations (below
$r=|x-x_{ij}\e|$):
\begin{gather}\label{comput}
\intl_{\partial
D_{ij}\e}\nu\e\left[{\mathbf{v}_{ij}\e}\right]dS\e=-\left.{\partial
\mathbf{v}_{ij}\e\over\partial
r}\right|_{r=d_j\e}(d_j\e)^{n-1}\omega_{n-1}\sim{1\over
2}\omega_{n-1}(n-2)d_j^{n-2}\eps^n=\sigma_j\rho_j\eps^n,\ \eps\to
0
\end{gather}
Since $\overline{C^\infty(\square)}=L_2(\square)$\textcolor{black}{,
then} $\xi\e$ converges weakly in $L_2(\square)$ to
$\ds\suml_{j=1}^m\sigma_j\rho_j(w-w_j)$ as $\eps\to 0$.

Using this, (\ref{pi_a}), (\ref{pi_b}) and (\ref{estim1}) we
conclude that
\begin{gather}
\label{estim2} \liml_{\eps\to 0}\intl_{\Omega_\square\e}-\Delta
w\e u\e dV\e=\intl_\square \left(-\Delta
wu+\suml_{j=1}^m\sigma_j\rho_j(w-w_j)u\right)dx
\end{gather}

In the same way (using the estimate (\ref{v_bf_estim+})) one can
prove that
\begin{gather}
\label{estim3} \liml_{\eps\to 0}\intl_{\Omega_\square\e}\lambda\e
w\e u\e dV\e=\intl_\square \lambda wu dx
\end{gather}

Now, we investigate the behaviour of the integrals in
(\ref{int_ineq}) over $\cupl_{i,j}B_{ij}\e$. Using
(\ref{v_bf_estim}) (the last asymptotics), (\ref{comput}) and the
Poincar\'{e} inequality we get
\begin{multline}
\suml_{i\in\I}\suml_{j=1}^m\intl_{B_{ij}\e}
-\Delta\Big[\big(w(x_{ij}\e)-w_j(x_{ij}\e)\big)\big(1-\mathbf{v}_{ij}\e(\tilde
x)\big)\Big]u\e(\tilde x)dV\e=\\=\suml_{i\in\I}\suml_{j=1}^m
\langle
u\e\rangle_{B_{ij}\e}\big(w(x_{ij}\e)-w_j(x_{ij}\e)\big)\intl_{\partial
D_{ij}\e}-\nu\e\left[\mathbf{v}_{ij}\e\right]dS\e+\bar
o(1)=\\\label{estim4}=\suml_{j=1}^m\sigma_j\rho_j\intl_{\square}
\widehat{[w_j-w]}(x)\Pi_j\e u\e(x) dx+\bar o(1)\underset{\eps\to
0}\to \suml_{j=1}^m\sigma_j\rho_j\intl_{\square}
(w_j(x)-w(x))u_j(x)dx
\end{multline}
where $\widehat{[w_j-w]}\in L_2(\square)$ is a step function:
$\widehat{[w_j-w]}(x)=w_j(x_{ij}\e)-w(x_{ij}\e)$,
$x\in\square_i\e$, $i\in\I$; it is clear that $\widehat{[w_j-w]}$
converges to $w_j-w$ strongly in $L_2(\square)$ as $\eps\to 0$.

In a similar manner we obtain
\begin{gather}
\label{estim5} \liml_{\eps \to
0}\suml_{i\in\I}\suml_{j=1}^m\intl_{B_{ij}\e}\lambda\e w\e u\e
dV\e=\lambda\suml_{j=1}^m\rho_j\intl_\square w_j u_j dx
\end{gather}

\textcolor{black}{Thus,} from (\ref{estim2})-(\ref{estim5}) we
obtain that the functions
$u\in\mathrm{dom}(\bar\eta_\square^\theta)$, $u_j\in L_2(\square)$
($j=1,\dots,m$) satisfy the equality:
\begin{multline}
 \intl_{\square}\left[-\Delta wu+\suml_{j=1}^m\sigma_j\rho_j
u(w-w_j)+ \suml_{j=1}^m\sigma_j\rho_j
u_j(w_j-w)\right]dx+\ds\intl_{\partial\square}\nu\left[w\right]u
ds-\\\label{result}-\lambda
\intl_{\square}\left[uw+\suml_{j=1}^m\rho_j u_jw_j\right]dx=0
\end{multline}
for \textcolor{black}{arbitrary} $w\in C_\theta^\infty(\square)$,
$w_j\in C^\infty(\square)$ ($j=1,\dots,m$).

Substituting $w\equiv 0$, $w_j\equiv 0$, $j\not= k$ into
(\ref{result}) we obtain
\begin{gather}\label{v=u}
u_k={\sigma_k u\over\sigma_k-\lambda},\ k=1,\dots,m
\end{gather}
Then substituting into (\ref{result}) $w_j\equiv 0$ ($\forall j$),
integrating by parts and taking into account
(\ref{v=u})\textcolor{black}{, we} conclude that
$u\in\mathrm{dom}(\bar\eta_{\square}^\theta)$ satisfies the
equality
\begin{gather*}
\intl_{\square}\Big[(\nabla u,\nabla
w)-\lambda\mathcal{F}(\lambda)uw\Big]dx=0,\quad \forall w\in
C_\theta^\infty(\square)
\end{gather*}
where $\mathcal{F}(\lambda)$ is defined by (\ref{mu_eq}). Hence
$u\in\mathrm{dom}(\Delta_{\square}^\theta)$ and
$$-\Delta_{\square}^\theta u=\lambda\mathcal{F}(\lambda)u$$
In view of Lemma \ref{lm3} $u\not=0$. Then
$\lambda\mathcal{F}(\lambda)\in\sigma(-\Delta_{\mathbb{R}^n})$
and\textcolor{black}{, therefore,} due to (\ref{arg})
$\lambda\in\sigma(\mathcal{A})
\setminus\cupl_{j=1}^m\left\{\sigma_j\right\}$.

The fulfilment of \textcolor{black}{property (\ref{ah})} is proved.

\subsection{\label{ss23}Property (\ref{bh}) of Hausdorff convergence}
Let $\lambda\in\sigma(\mathcal{A})\cap[0,L]$,
$L\not\in\cupl_{j=1}^m\left\{\mu_j \right\}$. We have to prove
that there exists $\lambda\e\in\sigma(-\Delta_{M\e})\cap[0,L]$
such that $\lambda\e\underset{\eps\to 0}\to\lambda$.

At first we prove \textcolor{black}{property (\ref{bh})}
\textcolor{black}{for the case} $\lambda<L$.

We assume the opposite: the subsequence (still denoted by $\eps$)
and $\delta>0$ exist such that
\begin{gather}
\label{dist} \dist(\lambda,\sigma(-\Delta_{M\e}))>\delta.
\end{gather}

Since $\lambda\in\sigma(\mathcal{A})$\textcolor{black}{,  then} the
function
$F=\left(\begin{matrix}f\\f_1\\\dots\\f_m\end{matrix}\right)\in H$
exists such that
\begin{gather}\label{notinim}
F\notin \mathrm{im}(\mathcal{A}-\lambda\mathrm{I}),\text{ where
}\mathrm{I}\text{ is the \textcolor{black}{identity operator}}\end{gather}

Let $f\e(\tilde x)\in L_2(M\e)$ be defined by the formula
\begin{gather*}
f\e(\tilde x)=\begin{cases} f(x),&\ \tilde x\in\Omega\e,\\\ds
\langle f_j\rangle_{\square_i\e},&\ \tilde x\in B_{ij}\e.
\end{cases}
\end{gather*}
It follows from \textcolor{black}{the Cauchy inequality} and
(\ref{b_size}) that the norms $\|f\e\|_{L_2(M\e)}$ are bounded
uniformly in $\eps$.

\textcolor{black}{Inequality} (\ref{dist}) implies that
$\lambda\in\mathbb{R}\setminus\sigma({-\Delta_{M\e}})$. Then
$\mathrm{im}(-\Delta_{M\e}-\lambda\mathrm{I})=L_2(M\e)$ and
thus\textcolor{black}{, the unique $u\e\in
\mathrm{dom}(\Delta_{M\e})$ exists satisfying}
\begin{gather}
\label{bvp1} -\Delta_{M\e} u\e-\lambda u\e=f\e
\end{gather}
In consequence of (\ref{dist}) $u\e$ satisfies the inequality
\begin{gather*}
\|u\e\|_{L_2(M\e)}\leq \delta^{-1}{\|f\e\|_{L_2(M\e)}}\leq C
\end{gather*}
Furthermore
\begin{gather*}
\|\nabla u\e\|^2_{L_2(M\e)}\leq \|f\e\|_{L_2(M\e)}\cdot
\|u\e\|_{L_2(M\e)}+|\lambda|\cdot \|u\e\|_{L_2(M\e)}^2\leq C
\end{gather*}

Then there exists a subsequence (still denoted by $\eps$) such
that
\begin{gather*}
\Pi\e u\e\rightarrow u\in H^1(\mathbb{R}^n)\text{ weakly in }
H^1(\mathbb{R}^n)\text{ and strongly in }L_2(G)\text{ for any compact set }G\subset\mathbb{R}^n\\
\Pi_j\e u\e\rightarrow u_j\in L_2(\mathbb{R}^n)\text{ weakly in }
L_2(\mathbb{R}^n)\ (j=1,\dots,m)
\end{gather*}
where $\Pi\e$, $\Pi_j\e$ ($j=1,\dots,m$) are the extension
operators introduced in the previous subsection.

For an arbitrary function $\mathbf{w}\e\in C^\infty_0(M\e)$ we
have
\begin{gather}\label{int_ineq_f}
\intl_{M\e}\left((\nabla\e u\e,\nabla\e \mathbf{w}\e)_\eps
-\lambda u\e \mathbf{w}\e-f\e \mathbf{w}\e\right) dV\e=0
\end{gather}
Let $w\in C_0^\infty(\mathbb{R}^n)$, $w_j\in
C_0^\infty(\mathbb{R}^n)$ ($j=1,\dots,m$) be
\textcolor{black}{arbitrary} functions. Using them we construct the
test-function $\mathbf{w}\e$ by \textcolor{black}{formula}
(\ref{w_bf}) (but with $\mathbb{R}^n$ instead of
$\Omega_\square\e$ and with $\mathbb{Z}^n$ instead of $\I$) and
substitute it into (\ref{int_ineq_f}). \textcolor{black}{Performing}
the same calculations as in the previous subsection we obtain
\begin{multline}
\intl_{\mathbb{R}^n}\left[(\nabla u,\nabla
w)+\suml_{j=1}^m\sigma_j\rho_ju(w-w_j)+\suml_{j=1}^m\sigma_j\rho_ju_j(w_j-w)-\right.\\\label{int_ineq_final}\left.-\lambda\left(uw+\suml_{j=1}^m
\rho_ju_j w_j\right)-\left(fw+\suml_{j=1}^m \rho_jf_j
w_j\right)\right]dx=0
\end{multline}
for \textcolor{black}{arbitrary} $w\in C_0^\infty(\mathbb{R}^n)$,
$w_j\in C_0^\infty(\mathbb{R}^n)$ ($j=1,\dots,m$). It follows from
(\ref{int_ineq_final}) that
\begin{gather*}
U=\left(\begin{matrix}u\\u_1\\\dots\\u_m\end{matrix}\right)\in\mathrm{dom}
(\mathcal{A})\text{\quad and\quad }\mathcal{A}U-\lambda U=F
\end{gather*}
We obtain a contradiction with (\ref{notinim}). Then there is
$\lambda\e\in\sigma(-\Delta_{M\e})$ such that $\liml_{\eps\to
0}\lambda\e=\lambda$. Since $\lambda<L$\textcolor{black}{, then}
$\lambda\e<L$ \textcolor{black}{when $\eps$ is small enough}.

\textcolor{black}{Finally,} we verify the fulfilment of
\textcolor{black}{property (\ref{bh})} for the case $\lambda=L$.
Since
$L\notin\cupl_{j=1}^m\left\{\mu_j\e\right\}$\textcolor{black}{,
then} (\ref{struct}) implies that $(L-\delta,L-\delta/2)\subset
\sigma(\mathcal{A})$ when $\delta$ is small enough. Let
$\lambda_\delta\in (L-\delta,L-\delta/2)$. \textcolor{black}{We have
just proved} that if $\eps<\eps(\delta)$ then
$\lambda^{\eps}\in\sigma(-\Delta_{M\e})$ exists such that
$|\lambda^{\eps}-\lambda_\delta|<\delta/2$. Then
$\lambda\e\in(L-3\delta/2,L)$ as $\eps<\eps(\delta)$ that
obviously implies the fulfilment of \textcolor{black}{property
(\ref{bh})}.

\subsection{\label{ss24}End of the proof}

\textcolor{black}{In the proof of} the Hausdorff convergence
\textcolor{black}{we used} the fact that $M\e$ is $\Gamma$-periodic
manifold, $\M\e$ is a period cell. Now let us recall that $M\e$ is
also $\Gamma\e$-periodic manifold, $\M_{i}\e$ is a corresponding
period cell ($i$ is arbitrary, \textcolor{black}{so from now on }we
consider $i=0$). Then
\begin{gather*}
\sigma(-\Delta_{M\e})=\cupl_{k=1}^\infty [a_k\e,b_k\e]
\end{gather*}
where $[a_k\e,b_k\e]=\left\{\lambda_k^\theta(\M_0\e),\
\theta\in\mathbb{T}^n\right\}$

\begin{lemma}\label{lm4}$\liml_{\eps\to 0}b_{m+1}\e=\infty$
\end{lemma}

\begin{proof}  As usual by $\lambda_k^N(\M_0\e)$ we denote the
$k$-th eigenvalue of the operator $-\Delta_{\M_0\e}^N$, which is
the Laplace-Beltrami operator on $\M_0\e$ with Neumann boundary
conditions.

Using the same idea as in the proof of Lemma \ref{lm0+} (i.e.
$\eps^{-1}$-homothetic image of $\M_0\e$)\textcolor{black}{, we
get}
\begin{gather}\label{relationship1}
\liml_{\eps\to 0}\eps^2\lambda_k(\M_0\e)=\lambda_k,\ k=1,2,3...
\end{gather}
where $\left\{\lambda_k\right\}_{k\in\mathbb{N}}$ are the
eigenvalues of the operator $\mathcal{L}$ which acts in the space
$L_2(\square)\underset{j=\overline{1,m}}\oplus L_2(\mathbf{B}_j)$
and is defined by the operation
\begin{gather*}
\mathcal{L}=-\left(\begin{matrix}\Delta^N_{\square}&0&...&0\\
0&\Delta_{\mathbf{B}_1}&...&0\\...\\0&0&...&\Delta_{\mathbf{B}_m}\end{matrix}\right)
\end{gather*}
Recall that $\square$ is the unit cube in $\mathbb{R}^n$,
$\mathbf{B}_j$ is the $n$-dimensional sphere of the radius $b_j$
($j=1,\dots,m$).

One has $\lambda_j=\lambda_1(\mathbf{B}_j)=0$, $j=1,\dots,m$,
$\lambda_{m+1}=\lambda_1^N(\square)=0$, and
$$\lambda_{m+2}=\minl\left\{\lambda^N_2(\square),\lambda_2(\mathbf{B}_j),\
j=1,\dots,m\right\}>0$$ \textcolor{black}{Thus,} in view of
(\ref{relationship1}) $\liml_{\eps\to
0}\lambda_{m+2}^N(\M_0\e)=\infty$. Due to
\textcolor{black}{inequality} (\ref{enclosure})
$\lambda_{m+2}^N(\M_0\e)\leq a_{m+2}\e$. \textcolor{black}{Thus,}
$\liml_{\eps\to 0}a_{m+2}\e=\infty$.

Suppose that there exists a subsequence (still denoted by $\eps$)
such that \textcolor{black}{the numbers $b_{m+1}\e$ are bounded}
uniformly in $\eps$. Let $L>\maxl_{j=\overline{1,m}}{\mu_j}$ and
$L>b_{m+1}\e$. Let $L_1>L$. Since $a_{m+2}\e\underset{\eps\to
0}\to \infty$\textcolor{black}{, then} $a_{m+2}\e>L_1$
\textcolor{black}{when $\eps$ is small enough}. Hence
$\sigma(-\Delta_{M\e})\cap [L,L_1]=\varnothing$
\textcolor{black}{when $\eps$ is small enough}. But this contradicts
to \textcolor{black}{property (\ref{bh})} of the Hausdorff
convergence. Hence $b_{m+1}\e\underset{\eps\to 0}\to \infty$.
\textcolor{black}{The lemma is proved.}
\end{proof}

It follows from Lemma \ref{lm4} that \textcolor{black}{within an
arbitrary finite interval $[0,L]$ the spectrum
$\sigma(-\Delta_{M\e})$ has at most $m$ gaps when $\eps$ is small
enough, i.e.}
\begin{gather}
\sigma(-\Delta_{M\e})\cap[0,L]=[0,L]\setminus
\cupl_{j=1}^{m\e}(\sigma_j\e,\mu_j\e)
\end{gather}
where $(\sigma_j\e,\mu_j\e)\subset[0,L]$ are some pairwise
disjoint intervals, $m\e\leq m$. Here we renumber the intervals in
the increasing order.

Let $L>\maxl_{j=\overline{1,m}}{\mu_j}$ be arbitrarily large
number. We have just proved that as $\eps\to 0$ the set
$\sigma(-\Delta_{M\e})\cap [0,L]$ converges to the set
$\sigma(\mathcal{A})\cap [0,L]=[0,L]
\setminus\left(\cupl_{j=1}^m(\sigma_j,\mu_j)\right)$ in the
Hausdorff sense. Then by Proposition \ref{prop1} $m\e=m$
\textcolor{black}{when $\eps$ is small enough} and
\begin{gather*}
\forall j=1,\dots,m:\quad \lim_{\eps\to
0}\sigma_j\e=\sigma_j,\quad \lim_{\eps\to 0}\mu_j\e=\mu_j
\end{gather*}

\textcolor{black}{Finally,} we denote by $\mathcal{J}\e$ the union
of the remaining gaps (if any). Since $b_{m+1}\e\underset{\eps\to
0}\to \infty$ and $b_{m+1}\e\leq
\inf\mathcal{J}\e$\textcolor{black}{, then} $$ \inf\mathcal{J}\e>L
$$ \textcolor{black}{when $\eps$ is small enough}. \textcolor{black}{This concludes
the proof of Theorem \ref{th2}}.

\begin{remark}\label{rem_strong}
\textcolor{black}{Actually, we have proved a slightly strong
result}: $\liml_{\eps\to 0}a_{k+1}\e= \mu_{k}$, $\liml_{\eps\to
0}b_{k}\e=\sigma_{k}$, $k=1,\dots,m$, $\liml_{\eps\to
0}b_{m+1}\e=\infty$, i.e. the first $m$ gaps of the spectrum
$\sigma(-\Delta_{M\e})$ ($\eps$ is small enough) are located
exactly between the first $(m+1)$ bands.
\end{remark}

\section{\label{sec4}End of the proof of Theorem \ref{th1}: choice of the constants $d_j$, $b_j$ and conclusive remarks}
In order to complete the proof of Theorem \ref{th1}, we have to
choose \textcolor{black}{the constants $d_j$, $b_j$ in
(\ref{d_size}), (\ref{b_size}) such that equalities (\ref{ab})
hold.}

%-------------------------------------------------------------------
%--------------THEOREM 2(CHOICE OF a_j AND b_j)---------------------
%-------------------------------------------------------------------
\begin{theorem}\label{th3}
Let $(\a_j,\b_j)$ ($j={1,\dots,m},\ m\in \mathbb{N} $) be
arbitrary intervals satisfying (\ref{intervals+}). Let $M\e$
($\eps>0$) be an $n$-dimensional periodic Riemannian manifolds of
the form (\ref{manifold}).

Then (\ref{ab}) holds if we choose
\begin{gather}\label{exact_formula1}
d_j=\begin{cases}\left[\ds{2(\b_j-\a_j)\over
\omega_{n-1}(n-2)}\prod\limits_{i=\overline{1,m}|i\not=
j}\ds\left({\b_i-\a_j\over  \a_i-\a_j}\right)\right]^{1\over
n-2},&n> 2\\\ds{(\b_j-\a_j)\over
\pi}\prod\limits_{i=\overline{1,m}|i\not=
j}\ds\left({\b_i-\a_j\over
\a_i-\a_j}\right),&n=2\end{cases}\\\label{exact_formula2}b_j=\left[\ds{{\b_j-\a_j\over
\omega_n\a_j}\prod\limits_{i=\overline{1,m}|i\not=
j}\ds\left({\b_i-\a_j\over  \a_i-\a_j}\right)}\right]^{1\over n}
\end{gather}
\end{theorem}

%-------------------------------------------------------------------
%--------------REMARK-----------------------------------------------
%-------------------------------------------------------------------

\begin{remark} Since the intervals $(\a_j,\b_j)$ satisfy
(\ref{intervals+})\textcolor{black}{, then}
\begin{gather*}
\forall j:\ \a_j<\b_j,\quad \forall i\not= j:\
\mathrm{sign}(\b_i-\a_j)=\mathrm{sign}(\a_i-\a_j)\not= 0
\end{gather*}
\textcolor{black}{Therefore,} the expressions
$(\b_j-\a_j)\prod\limits_{i=\overline{1,m}|i\not=
j}\ds\left({\b_i-\a_j\over  \a_i-\a_j}\right),\ j=1,\dots,m$ are
positive and thus the choice of $d_j$ and $b_j$ is
correct.\end{remark}

\begin{proof} Substituting $d_j$, $b_j$ (\ref{exact_formula1}),
(\ref{exact_formula2}) into (\ref{sigma}) we get
$$\sigma_j=\alpha_j$$ i.e. the first equality in (\ref{ab}) holds.
 Furthermore substituting $b_j$
(\ref{exact_formula2}) into (\ref{rho}) we obtain
\begin{gather}\label{prod}
\rho_j={\b_j-\a_j\over\a_j}\prod\limits_{i=\overline{1,m}|i\not=
j}\ds\left({\b_i-\a_j\over  \a_i-\a_j}\right)
\end{gather}

It remains to prove that $\mu_j=\b_j$. Recall that the numbers
$\mu_j$ ($j=1,\dots,m$) are the roots of the equation
(\ref{mu_eq}). \textcolor{black}{Therefore,} in order to prove the
equality $\mu_j=\b_j$\textcolor{black}{, we have} to show that
\begin{gather}\label{system}
\forall k=1,\dots ,m:\ \suml_{j=1}^m{\a_j\rho_j\over \b_k-\a_j}=1
\end{gather}

Let us consider (\ref{system}) as the linear algebraic system of
$m$ equations with unknowns $\rho_j$ ($j=1,\dots,m$). In order to
end the proof of theorem we have to prove the following
\begin{lemma}\label{lm1}
The system (\ref{system}) has the unique solution
$\rho_1,\dots,\rho_m$ which is defined by (\ref{prod}).
\end{lemma}
\begin{proof}  We prove the lemma by induction. For $m=1$ its
validity is obvious. Suppose that we have proved it for $m=N-1$
and let us prove it for $m=N$.

Multiplying the $k$-th equation in (\ref{system}) ($k=1,\dots,N$)
by $\b_k-\a_N$ and then subtracting the $N$-th equation from the
first $N-1$ equations we obtain
\begin{gather*}
\forall k=1,\dots, N-1:\ \suml_{j=1}^{N-1}{\a_j\hat\rho_j\over
\b_k-\a_j}=1
\end{gather*}
where the new variables $\hat\rho_j$, $j=1,\dots, N-1$ are
expressed in terms of $\rho_j$ by the formula
\begin{gather}\label{new}
\hat\rho_j:=\rho_j\ds{\a_N-\a_j\over \b_N-\a_j},\ j=1,\dots, N-1
\end{gather}
Thus\textcolor{black}{, the numbers $\hat\rho_j$} satisfy the system
(\ref{system}) with $m=N-1$. \textcolor{black}{Therefore,} by the
induction
\begin{gather}\label{syst_sol_ind}
\hat\rho_j={\b_j-\a_j\over\a_j}\prod\limits_{i=\overline{1,N-1}|i\not=
j}\ds\left({\b_i-\a_j\over  \a_i-\a_j}\right)
\end{gather}
It follows from (\ref{new}), (\ref{syst_sol_ind}) that $\rho_j$,
$j=1,\dots,N-1$, satisfy \textcolor{black}{formula} (\ref{prod}).
The validity of this formula for $\rho_N$ follows from the
symmetry of the system.

Lemma \ref{lm1} and Theorem \ref{th3} are proved. This completes
the proof of the main theorem.

\end{proof}\end{proof}

\begin{remark}\label{rem22}We noted above that the \textcolor{black}{metric} $g\e$ of the
manifold $M\e$ is continuous but piecewise-smooth (see
\textcolor{black}{formulae} (\ref{metrics})-(\ref{metrics+})).
However one can approximate $g\e$ by a smooth
\textcolor{black}{metric} $g^{\eps\rho}$ which differs from $g\e$
only in small $\rho$-neighbourhoods of $\partial B_{ij}\e$ and
moreover the corresponding Laplace-Beltrami operator has the same
spectral properties as $\eps\to 0$.

Namely, in a small neighbourhood $U_{ij}\e$ of $\partial B_{ij}\e$
we introduce the local coordinates $(x_1,\dots,x_n)$ by
\textcolor{black}{formulae} (\ref{coord}) and define $g^{\eps\rho}$
by the formula
$$g_{\a\b}^{\eps\rho}(x_1,\dots,x_n)=g_{+_{\a\b}}^{\eps}(x_1,\dots,x_n)\varphi(x_n/\rho)+
g_{-_{\a\b}}^{\eps}(x_1,\dots,x_n)(1-\varphi(x_n/\rho))$$ where
$\phi(r)$, $r\in\mathbb{R}$, is a smooth positive function equal
to $1$ as $r\geq 1$, equal to $0$ as $r\leq -1$ and positive as
$-1<r<1$, the coefficients $g_{\pm_{\a\b}}^{\eps}$ are defined by
(\ref{metrics+}). Outside $\cupl_{i,j}U_{ij}\e$ we set
$g^{\eps\rho}=g\e$.

It is easy to see that $ A^{\eps\rho}g^{\eps}\leq g^{\eps\rho}\leq
B^{\eps\rho}g^{\eps} $, where $A^{\eps\rho},B^{\eps\rho}$ are
positive constants depending on $\eps$ and $\rho$ in such a way
that for \textbf{fixed} $\eps$
\begin{gather}\label{AB-in1}
\liml_{\rho\to 0}A^{\eps\rho}=\liml_{\rho\to 0}B^{\eps\rho}=1
\end{gather}
Using the min-max principle one can obtain that
\begin{gather}\label{AB-in2}
\forall k\in\mathbb{N},\ \forall\theta \in\mathbb{T}^n:\quad
{(A^{\eps\rho})^{n/2}\over
(B^{\eps\rho})^{1+n/2}}\lambda_k^\theta(\M_i\e)\leq
\lambda_k^\theta(\M_i\e,g^{\eps\rho})\leq{(B^{\eps\rho})^{n/2}\over
(A^{\eps\rho})^{1+n/2}}\lambda_k^\theta(\M_i\e)
\end{gather}
Here $\lambda_k^\theta(\M_i\e,g^{\eps\rho})$ is the $k$-th
eigenvalue of the Laplace-Beltrami operator with $\theta$-periodic
boundary conditions on the manifold $\M_i\e$ equipped with the
\textcolor{black}{metric} $g^{\eps\rho}$. This inequality is proved
in \cite[Chapter A]{Anne} for manifolds without a boundary, for
our case the proof is \textcolor{black}{completely analogous}.

Let $\delta_1>0,\ L_1>0$. We have just proved (see Theorems
\ref{th2}, \ref{th3}) that there are such
$\eps=\eps(\delta_1,L_1)$ and such $d_j,b_j$ that the manifold
$M=M\e$ satisfies (\ref{spec1})-(\ref{spec2}) with
$\delta=\delta_1$, $L=L_1$.

So let us fix $\eps=\eps(L_1,\delta_1)$. Then it follows from
(\ref{AB-in1}), (\ref{AB-in2}) that
\begin{gather}\label{AB-in3}
\forall\theta\in\mathbb{T}^n,\ \forall k\in\mathbb{N}:\
\left|\lambda_k^\theta(\M_i\e)-
\lambda_k^\theta(\M_i\e,g^{\eps\rho})\right| \underset{\rho\to
0}\rightarrow 0
\end{gather}
uniformly in $(\theta,k)$ from $\mathbb{T}^n\times \mathbb{G}$,
where $\mathbb{G}$ is any compact subset of $\mathbb{N}$. Then
using (\ref{fund}), (\ref{AB-in3}) and taking into account Remark
\ref{rem_strong} we conclude: there is such
$\rho=\rho(\eps(\delta_1,L_1))$ that the manifold
$(M\e,g^{\eps\rho})$ satisfies (\ref{spec1})-(\ref{spec2}) with
$\delta=2\delta_1$, $L=L_1-\delta_1$.

Now, let $\delta>0$, $L>0$. Setting $\delta_1=\delta/2$,
$L_1=L+\delta/2$ we conclude that the manifold
$M=(M\e,g^{\eps\rho})$, where $\eps=\eps(\delta_1,L_1)$,
$\rho=\rho(\eps(\delta_1,L_1))$, satisfies
(\ref{spec1})-(\ref{spec2}).
\end{remark}

\section*{Acknowledgements} The author is grateful to Prof. E. Khruslov
for the helpful discussion. The work is supported by the
French-Ukrainian grant "PICS 2009-2011. Mathematical
Physics:Methods and Applications".

\end{document}